\documentclass[11pt]{article}
\usepackage{amsfonts, amssymb, graphicx, amsmath}
\usepackage{fourier}
\usepackage{fullpage}
\usepackage{natbib}
\usepackage{float}
\usepackage[usenames]{color}
\usepackage[pdftex]{hyperref}
\hypersetup{
    colorlinks=true,       
    linkcolor=blue,        
    citecolor=red,        
    filecolor=blue,        
    urlcolor=blue          
}
\RequirePackage{natbib}
\usepackage{amsmath,amssymb,latexsym,amsthm}
\usepackage{graphicx}
\usepackage[ruled,vlined]{algorithm2e}

\newtheorem{theorem}{Theorem}
\newtheorem{lemma}{Lemma}

\newtheorem{remark}{Remark}[theorem]

\let\inf\undef
\DeclareMathOperator*{\inf}{\vphantom{p}inf}
\let\sup\undef
\DeclareMathOperator*{\sup}{\vphantom{p}sup}

\newcommand{\mbb}[1]{\mathbb{#1}}

\newcommand{\mc}[1]{\mathcal{#1}}
\newcommand{\mrm}[1]{\mathrm{#1}}

\newcommand{\KL}{d_{\sf KL}}

\newcommand{\sign}{\mrm{sign}}
\newcommand{\argmin}[1]{\underset{#1}{\mrm{argmin}} \ }

\begin{document}

\title{Computational and Statistical Boundaries for Submatrix Localization in a Large Noisy Matrix}
\author{T. Tony Cai\footnote{The research of Tony Cai was supported in part by NSF Grants DMS-1208982 and DMS-1403708,  and NIH Grant R01 CA127334.}, \;\, Tengyuan Liang\footnote{Tengyuan Liang acknowledges the support of Winkelman Fellowship.} \; and \; Alexander Rakhlin\footnote{Alexander Rakhlin gratefully acknowledges the support of NSF
under grant CAREER DMS-0954737.}\\ \\
Department of Statistics\\ The Wharton School \\ University of Pennsylvania}
\date{}
\maketitle

\begin{abstract}

The interplay between computational efficiency and statistical accuracy in high-dimensional inference has drawn increasing attention in the literature.
In this paper, we study computational and statistical boundaries for submatrix localization. Given one observation of (one or multiple non-overlapping) signal submatrix (of magnitude $\lambda$ and size $k_m \times k_n$) contaminated with a noise matrix (of size $m \times n$), we establish two transition thresholds for the signal to noise $\lambda/\sigma$ ratio in terms of $m$, $n$, $k_m$, and $k_n$. The first threshold, $\sf SNR_c$, corresponds to the computational boundary. Below this threshold, it is shown that no polynomial time algorithm can succeed in identifying the submatrix, under the \textit{hidden clique hypothesis}.  We introduce adaptive linear time spectral algorithms that identify the submatrix with high probability when the signal strength is above the threshold $\sf SNR_c$.   The second threshold, $\sf SNR_s$, captures the statistical boundary, below which no method can succeed with probability going to one in the minimax sense. The exhaustive search method successfully finds the submatrix above this threshold. The results show an interesting phenomenon that $\sf SNR_c$ is always significantly larger than $\sf SNR_s$, which implies an essential gap between statistical optimality and computational efficiency for submatrix localization. 
\end{abstract}


\section{Introduction}


The ``signal + noise" model 
\begin{align}
X = M + Z, 
\end{align}
where $M$ is the signal of interest and $Z$ is noise, is ubiquitous in statistics and is used in a wide range of applications. When $M$ and $Z$ are matrices, many interesting problems arise under a variety of structural assumptions on $M$ and the distribution of $Z$. Examples include sparse principal component analysis (PCA) \citep{vu2012minimax,berthet2013optimal,birnbaum2013minimax,cai2013sparse,cai2015optimal}, non-negative matrix factorization \citep{lee2001algorithms}, non-negative PCA \citep{zass2006nonnegative,montanari2014non}. Under the conventional statistical framework, one is looking for optimal statistical procedures for recovering the signal or detecting its presence. 

As the dimensionality of the data becomes large, the computational concerns associated with statistical procedures 
come to the forefront. In particular, problems with a combinatorial structure or non-convex constraints pose a significant computational challenge because naive methods based on exhaustive search are typically not computationally efficient. Trade-off between computational efficiency and statistical accuracy in  high-dimensional inference has drawn increasing attention in the literature. In particular, \cite{chandrasekaran2012convex} and \cite{wainwright2014structured} considered a general class of linear inverse problems, with different emphasis on convex geometry and decomposition of statistical and computational errors. \cite{chandrasekaran2013computational} studied an approach for trading off computational demands with statistical accuracy via relaxation hierarchies. \cite{berthet2013computational,ma2013computational,zhang2014lower} focused on computational requirements for various statistical problems, such as detection and regression. 

In the present paper, we study the interplay between computational efficiency and statistical accuracy in  submatrix localization based on a noisy observation of a large matrix. The problem considered in this paper is formalized as follows.

\subsection{Problem Formulation}

Consider the matrix $X$ of the form
\begin{align}
	\label{submat.Mat.Frm}
	X = M + Z, ~~ \text{where} ~~ M = \lambda \cdot 1_{R_m} 1_{C_n}^T
\end{align}
and $1_{R_m}\in \mathbb{R}^m$ with $1$ on the index set $R_m$ and zero otherwise. Here, the entries $Z_{ij}$ of the noise matrix are i.i.d. zero-mean sub-Gaussian random variables with parameter $\sigma$ (defined formally in Equation \eqref{SubG.Def}).
Given the parameters $m,n,k_m,k_n,\lambda/\sigma$, the set of all distributions described above---for all possible choices of $R_m$ and $C_n$---forms the submatrix model $\mathcal{M}(m,n,k_m,k_n,\lambda/\sigma)$. 

This model can be further extended to the case of multiple submatrices as
\begin{align}
	\label{multmat.Mat.Frm}
	M = \sum_{s= 1}^r \lambda_s \cdot 1_{R_s} 1_{C_s}^T
\end{align}
where $|R_s| =  k_s^{(m)}$ and $|C_s| = k_s^{(n)}$ denote the support set of the $s$-th submatrix. 
For simplicity, we first focus on the single submatrix and then extend the analysis to the model \eqref{multmat.Mat.Frm} in Section \ref{sec:ext}.

There are two fundamental questions associated with the submatrix model \eqref{submat.Mat.Frm}. One is the \textit{detection} problem: given one observation of the $X$ matrix, decide whether it is generated from a distribution in the submatrix model or from the pure noise model. Precisely, the detection problem considers testing of the hypotheses 
$$
H_0 : M = {\bf 0} ~~~ \text{v.s.} ~~~ H_\alpha: M \in \mathcal{M}(m,n,k_m,k_n,\lambda/\sigma).
$$
The other is the \textit{localization} problem, where the goal is to exactly recover the signal index sets $R_m$ and $C_n$ (the support of the mean matrix $M$). It is clear that the localization problem is at least as hard (both computationally and statistically) as the detection problem. As we show in this paper, the localization problem requires larger signal to noise ratio $\lambda/\sigma$, as well as a more detailed exploitation of the submatrix structure. 

If the signal to noise ratio is sufficiently large, it is computationally easy to localize the submatrix. On the other hand, if this ratio is small, the localization problem is statistically impossible. To quantify this phenomenon, we identify two distinct thresholds ($\sf SNR_s$ and $\sf SNR_c$) for $\lambda/\sigma$ in terms of parameters $m, n, k_m, k_n$. The first threshold, $\sf SNR_s$, captures the statistical boundary, below which no method (possibly exponential time) can succeed with probability going to one in the minimax sense. The exhaustive search method successfully finds the submatrix above this threshold. 
The second threshold, $\sf SNR_c$, corresponds to the computational boundary, above which an adaptive (with respect to the parameters) linear time spectral algorithm finds the signal. Below this threshold, no polynomial time algorithm can succeed, under the \textit{hidden clique hypothesis}, described later.

\subsection{Prior Work}

There is a growing body of work in statistical literature on submatrix problems. 
\cite{shabalin2009finding} provided a fast iterative maximization algorithm to solve the submatrix localization problem. However, as with many EM type algorithms, the theoretical result is very sensitive to initialization. \cite{arias2011detection} studied the detection problem for a cluster inside a large matrix. \cite{butucea2013detection,butucea2013sharp} formulated the submatrix detection and localization problems under Gaussian noise and determined sharp statistical transition boundaries. For the detection problem, \cite{ma2013computational} provided a computational lower bound result under the assumption that hidden clique detection is computationally difficult. 

\cite{balakrishnan2011statistical, kolar2011minimax} focused on statistical and computational trade-offs for the submatrix localization problem. They provided a computationally feasible entry-wise thresholding algorithm, a row/column averaging algorithm, and a convex relaxation for sparse SVD to investigate the minimum signal to noise ratio that is required in order to localize the submatrix. Under the sparse regime $k_m \precsim m^{1/2}$ and $k_n \precsim n^{1/2}$, the entry-wise thresholding turns out to be the ``near optimal'' polynomial-time algorithm (which we will show a de-noised spectral algorithm that perform slightly better in Section~\ref{subsec:sparse}). However, for the dense regime when $k_m \succsim m^{1/2}$ and $k_n \succsim n^{1/2}$, the algorithms provided in \cite{kolar2011minimax} are not optimal in the sense that there are other polynomial-time algorithm that can succeed in finding the submatrix with smaller SNR. Concurrently with our work, \cite{chen2014statistical} provided a convex relaxation algorithm that improves the SNR boundary of \cite{kolar2011minimax} in the dense regime. On the downside, the implementation of the method requires a full SVD on each iteration, and therefore does not scale well with the dimensionality of the problem. Furthermore, there is no computational lower bound in the literature to guarantee the optimality of the SNR boundary achieved in \cite{chen2014statistical}. 

A problem similar to submatrix localization is that of clique finding. \cite{deshpande2013finding} presented an iterative approximate message passing algorithm to solve the latter problem with sharp boundaries on SNR. However, in contrast to submatrix localization, where the signal submatrix can be located anywhere within the matrix, the clique finding problem requires the signal to be centered on the diagonal.

We would like to emphasize the difference between detection and localization problems. When $M$ is a vector, \cite{donoho2004higher} proposed the ``higher criticism'' approach to solve the detection problem under the Gaussian sequence model. Combining the results in \citep{donoho2004higher,ma2013computational}, in the computationally efficient region, there is no loss in treating $M$ in model \eqref{submat.Mat.Frm} as a vector and applying the higher criticism method to the vectorized matrix for the problem of submatrix detection. In fact, the procedure achieves sharper constants in the Gaussian setting. However, in contrast to the detection problem, we will show that for localization, it is crucial to utilize the matrix structure, even in the computationally efficient region.

\subsection{Notation}
Let $[m]$ denote the index set $\{1,2,\ldots,m\}$. For a matrix $X \in \mathbb{R}^{m\times n}$, $X_{i \cdot} \in \mathbb{R}^n$ denotes its $i$-th row and $X_{\cdot j} \in \mathbb{R}^m$ denotes its $j$-th column. For any $I \subseteq [m], J \subseteq [n]$, $X_{IJ}$ denotes the submatrix corresponding to the index set $I \times J$. 
For a vector $v \in \mathbb{R}^n$, $\| v \|_{\ell_p} = (\sum_{i \in [n]} |v_i|^p)^{1/p}$ and for a matrix $M \in \mathbb{R}^{m \times n}$, $\| M \|_{\ell_p} = \sup_{v \neq \bf{0}} \| M v \|_{\ell_p}/\| v \|_{\ell_p}$. When $p=2$, the latter is the usual spectral norm, abbreviated as $\| M \|_{2}$. The nuclear norm a matrix $M$ is defined as a convex surrogate for the rank, with the notation to be $\| M \|_*$. The Frobenius norm of a matrix $M$ is defined as $\| M \|_{F} = \sqrt{\sum_{i,j}M_{ij}^2}$. The inner product associated with the Frobenius norm is defined as $\langle A, B \rangle = {\sf tr}(A^T B)$. 

Denote the asymptotic notation $a(n) = \Theta(b(n))$ if there exist two universal constants $c_l,c_u$ such that $c_l \leq \varliminf\limits_{n\rightarrow \infty} a(n)/b(n) \leq \varlimsup\limits_{n\rightarrow \infty} a(n)/b(n) \leq c_u$. $\Theta^*$ is asymptotic equivalence hiding logarithmic factors in the following sense: $a(n) = \Theta^*(b(n))$ iff there exists $c>0$ such that $a(n) = \Theta(b(n) \log^c n)$. Additionally, we use the notation $a(n) \asymp b(n)$ as equivalent to $a(n) = \Theta(b(n))$, $a(n) \succsim b(n)$ iff $\lim_{ n\rightarrow \infty} a(n)/b(n) = \infty$ and $a(n) \precsim b(n)$ iff $\lim_{n \rightarrow \infty} a(n)/b(n) = 0$. 

We define the zero-mean sub-Gaussian random variable $\bf z$ with sub-Gaussian parameter $\sigma$ in terms of its Laplacian. If there exists a universal constant $c>0$,
\begin{align}
	\label{SubG.Def}
	\mbb{E} e^{\lambda {\bf z}} \leq \exp(\sigma^2 \lambda^2/2c), ~~ \text{for all}~~ \lambda > 0
\end{align}
then we have
$$
\mbb{P}(|{\bf z}| > \sigma t) \leq 2 \cdot \exp(-c \cdot t^2/2).
$$
We call a random vector $Z \in \mathbb{R}^n$ isotropic with parameter $\sigma$ if
$$
\mathbb{E} (v^T Z)^2 = \sigma^2 \| v \|_{\ell_2}^2, ~~\text{for all}~~ v \in \mathbb{R}^n.
$$
Clearly, Gaussian and Bernoulli measures, and more general product measures of zero-mean sub-Gaussian random variables satisfy this isotropic definition up to a constant scalar factor.

\subsection{Our Contributions}

\label{sec:contribution}
To state our main results, let us first define a hierarchy of algorithms in terms of their worst-case running time on instances of the submatrix localization problem:
	\begin{align*}
		{\sf LinAlg} \subset {\sf PolyAlg} \subset {\sf ExpoAlg} \subset {\sf AllAlg}.
	\end{align*}
The set ${\sf LinAlg}$ contains algorithms $\mathcal{A}$ that produce an answer (in our case, the localization subset $\hat{R}^{\mathcal{A}}_m, \hat{C}^{\mathcal{A}}_n$) in time linear in $m\times n$ (the minimal computation required to read the matrix). The classes ${\sf PolyAlg}$ and ${\sf ExpoAlg}$ of algorithms, respectively, terminate in polynomial and exponential time, while ${\sf AllAlg}$ has no restriction.

Combining Theorem \ref{thm:comp-bdr} and \ref{thm:comp-bdr-sparse} in Section \ref{sec:comp-bdr} and Theorem \ref{thm:stat-bdr} in Section \ref{sec:stat-bdr}, the statistical and computational boundaries for submatrix localization can be summarized as follows.
	
\begin{theorem}[Computational and Statistical Boundaries]
	\label{thm:main}
	Consider the submatrix localization problem under the model \eqref{submat.Mat.Frm}. 
	The computational boundary ${\sf SNR_c}$ for the dense case when $\min\{k_m,k_n\} \succsim \max\{m^{1/2},n^{1/2}\}$ is
		\begin{align}
			{\sf SNR_c} \asymp  \sqrt{\frac{m \vee n}{k_m k_n}} + \sqrt{\frac{\log n}{k_m} \vee \frac{\log m}{k_n}},
		\end{align}
		in the sense that
		\begin{align}
			\label{eq:comp-upp}
			&\varlimsup_{m,n,k_m,k_n\rightarrow \infty} \inf_{\mathcal{A} \in {\sf LinAlg}}~ \sup_{M \in \mathcal{M}} \mathbb{P}\left( \hat{R}^{\mathcal{A}}_m \neq R_m  \text{ or }   \hat{C}^{\mathcal{A}}_n \neq C_n \right) = 0, &  \text {if } \; \frac{\lambda}{\sigma} \succsim {\sf SNR_c} \\
			\label{eq:comp-low}
			&\varliminf_{m,n,k_m,k_n\rightarrow \infty} \inf_{\mathcal{A} \in {\sf PolyAlg}}~ \sup_{M \in \mathcal{M}} \mathbb{P}\left( \hat{R}^{\mathcal{A}}_m \neq R_m  \text{ or }   \hat{C}^{\mathcal{A}}_n \neq C_n \right) > 0, &  \text {if } \;  \frac{\lambda}{\sigma} \precsim {\sf SNR_c}
		\end{align}
		where \eqref{eq:comp-low} holds under the Hidden Clique hypothesis ~$\sf HC_{l}$ (see Section~\ref{subsec:complwd}). 
		For the sparse case when $\max\{k_m, k_n\} \precsim \min\{ m^{1/2}, n^{1/2}\}$, the computational boundary is ${\sf SNR_c}  = \Theta^*(1)$, more precisely
		$$
		1\precsim {\sf SNR_c} \precsim \sqrt{\log \frac{m\vee n}{k_m k_n}}.
		$$
		
	The statistical boundary ~${\sf SNR_s}$ is
	\begin{align}
		{\sf SNR_s} \asymp  \sqrt{\frac{\log n}{k_m} \vee \frac{\log m}{k_n}},
	\end{align}	
	in the sense that
	\begin{align}
		\label{eq:stat-upp}
		&\varlimsup_{m,n,k_m,k_n\rightarrow \infty} \inf_{\mathcal{A} \in {\sf ExpoAlg}} ~\sup_{M \in \mathcal{M}} \mathbb{P}\left( \hat{R}^{\mathcal{A}}_m \neq R_m  \text{ or }  \hat{C}^{\mathcal{A}}_n \neq C_n \right) = 0, &  \text {if } \; \frac{\lambda}{\sigma} \succsim {\sf SNR_s} \\
		\label{eq:stat-low}
		&\varliminf_{m,n,k_m,k_n\rightarrow \infty} \inf_{\mathcal{A} \in {\sf AllAlg}} ~\sup_{M \in \mathcal{M}} \mathbb{P}\left( \hat{R}^{\mathcal{A}}_m \neq R_m  \text{ or }   \hat{C}^{\mathcal{A}}_n \neq C_n \right) > 0, &  \text {if } \; \frac{\lambda}{\sigma} \precsim {\sf SNR_s}
	\end{align}
	under the minimal assumption $\max\{ k_m,k_n \} \precsim \min\{ m,n \}$. 
\end{theorem}

If we parametrize the submatrix model as $m = n, k_m \asymp k_n \asymp k = \Theta^*(n^{\alpha}), \lambda/\sigma = \Theta^*(n^{-\beta})$, for some  $0<\alpha,\beta<1$, we can summarize the results of Theorem~\ref{thm:main} in a phase diagram, as illustrated in Figure \ref{fig:phase}.
    \begin{figure}[pht]
        \centering
        \includegraphics[width=5in]{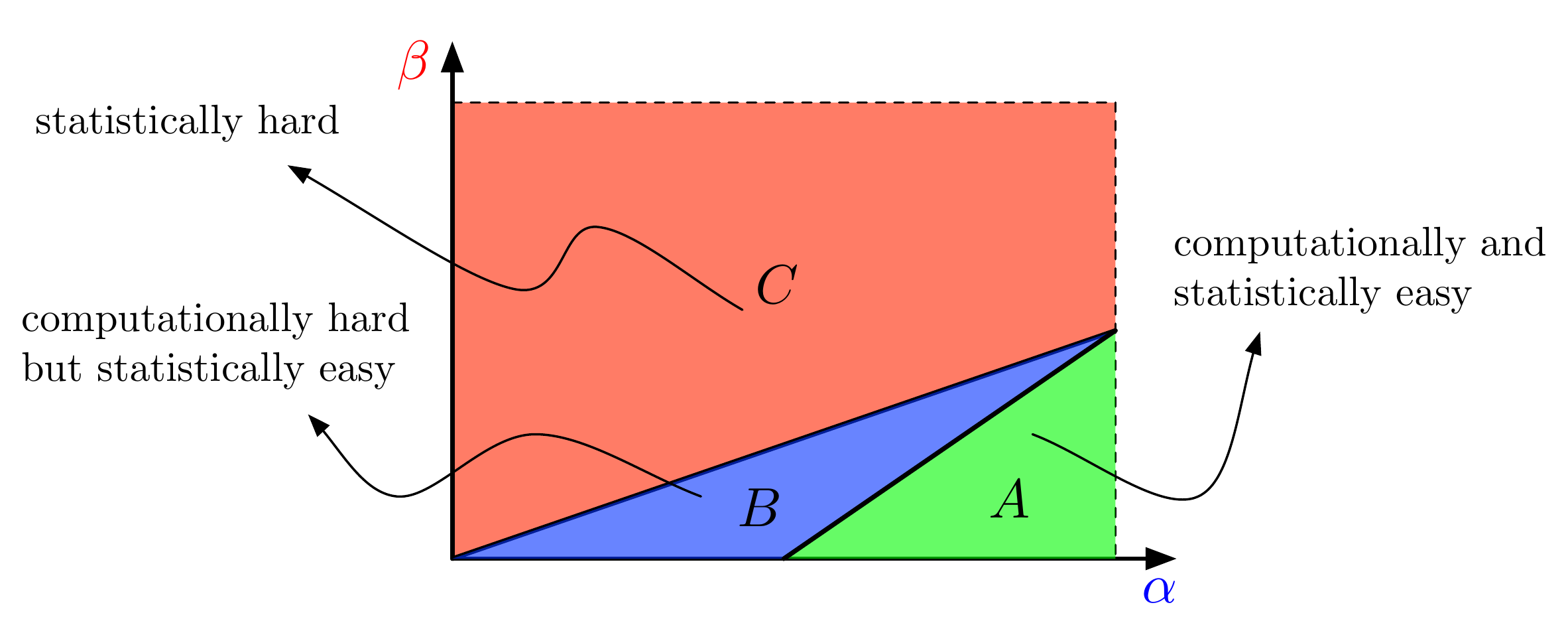}
		\caption{Phase diagram for submatrix localization. Red region (C): statistically impossible, where even without computational budget, the problem is hard. Blue region (B): statistically possible but computationally expensive (under the hidden clique hypothesis), where the problem is hard to all polynomial time algorithm but easy with exponential time algorithm. Green region (A): statistically possible and computationally easy, where a fast polynomial time algorithm will solve the problem.}
		\label{fig:phase}
    \end{figure}
    
To explain the diagram, consider the following cases. First, the statistical boundary is $$ \sqrt{\frac{\log n}{k_m} \vee \frac{\log m}{k_n}},$$ which gives the line separating the red and the blue regions. For the dense regime $\alpha \geq 1/2$, the computational boundary given by Theorem~\ref{thm:main} is $$\sqrt{\frac{m \vee n}{k_m k_n}} + \sqrt{\frac{\log n}{k_m} \vee \frac{\log m}{k_n}},$$ which corresponds to the line separating the blue and the green regions. For the sparse regime $\alpha < 1/2$,  the computational boundary is $\Theta(1) \precsim {\sf SNR_c}  \precsim \Theta(\sqrt{\log \frac{m\vee n}{k_m k_n}})$, which is the horizontal line connecting $(\alpha = 0, \beta = 0)$ to $(\alpha = 1/2, \beta = 0)$.

    As a key part of Theorem~\ref{thm:main}, we provide various linear time spectral algorithms that will succeed in localizing the submatrix with high probability in the regime above the computational threshold. Furthermore, the method is adaptive: it does not require the prior knowledge of the size of the submatrix. This should be contrasted with the method of \cite{chen2014statistical} which requires the prior knowledge of $k_m,k_n$; furthermore, the running time of their SDP-based method is superlinear in $nm$. Under the hidden clique hypothesis, we prove that below the computational threshold there is no polynomial time algorithm that can succeed in localizing the submatrix. This is a new result that has not been established in the literature. 
	We remark that the computational lower bound for \emph{localization} requires a technique different from the lower bound for \emph{detection}; the latter has been resolved in \cite{ma2013computational}.  
	
	Beyond localization of one single submatrix, we generalize both the computational and statistical story to a growing number of submatrices in Section~\ref{sec:ext}. As mentioned earlier, the statistical boundary for one single submatrix localization has been investigated by \cite{butucea2013sharp} in the Gaussian case. Our result focuses on the computational intrinsic difficulty of localization for a growing number of submatrices, at the expense of not providing the exact constants for the thresholds. 	
    
    The phase transition diagram in Figure~\ref{fig:phase} for localization should be contrasted with the corresponding result for detection, as shown in \citep{butucea2013detection, ma2013computational}. For a large enough submatrix size (as quantified by $\alpha>2/3$), the computationally-intractable-but-statistically-possible region collapses for the detection problem, but not for localization. In plain words, detecting the presence of a large submatrix becomes both computationally and statistically easy beyond a certain size, while for localization there is always a gap between statistically possible and computationally feasible regions. This phenomenon also appears to be distinct to that of other problems like estimation of sparse principal components \citep{cai2013sparse}, where computational and statistical easiness coincide with each other over a large region of the parameter spaces.

\subsection{Organization of the Paper}

The paper is organized as follows. Section \ref{sec:comp-bdr} establishes the computational boundary, with the computational lower bounds given in Section \ref{subsec:complwd} and upper bound results in Sections \ref{subsec:alg}-\ref{subsec:sparse}. 
An extension to the case of multiple submatrices is presented in Section \ref{sec:ext}.  
The upper and lower bounds for statistical boundary for multiple submatrices are discussed in Section \ref{sec:stat-bdr}. A  discussion is given in Section \ref{sec:dis}. Technical proofs are deferred to Section \ref{sec:pf}. In addition to the spectral method given in Section \ref{subsec:alg} and \ref{subsec:sparse}, Appendix \ref{sec:conv} contains a new analysis of a known method that is based on a convex relaxation \citep{chen2014statistical}. Comparison of computational lower bounds for localization and detection is included in Appendix \ref{sec:complow-det}.

\section{Computational Boundary}
\label{sec:comp-bdr}

We characterize in this section  the computational boundaries for the submatrix localization problem.  Sections  \ref{subsec:complwd}  and \ref{subsec:alg} consider respectively the computational lower bound and upper bound. The computational lower bound given in Theorem \ref{CLL.Thm} is based on the hidden clique hypothesis.

\subsection{Algorithmic Reduction and Computational Lower Bound}
\label{subsec:complwd}

Theoretical Computer Science identifies a range problems which are believed to be ``hard,'' in the sense that in the worst-case the required computation grows exponentially with the size of the problem. Faced with a new computational problem, one might try to reduce any of the ``hard'' problems to the new problem, and therefore claim that the new problem is as hard as the rest in this family. Since statistical procedures typically deal with a random (rather than worst-case) input, it is natural to seek token problems that are believed to be computationally difficult on average with respect to some distribution on instances. The hidden clique problem is one such example (for recent results on this problem, see \cite{feldman2013statistical,deshpande2013finding}). While there exists a quasi-polynomial algorithm, no polynomial-time method (for the appropriate regime, described below) is known. Following several other works on reductions for statistical problems, we work under the hypothesis that no polynomial-time method exists.

Let us make the discussion more precise. Consider the hidden clique model $\mc{G}(N, \kappa)$ where $N$ is the total number of nodes and $\kappa$ is the number of clique nodes. In the hidden clique model, a random graph instance is generated in the following way. Choose $\kappa$ clique nodes uniformly at random from all the possible choices, and connect all the edges within the clique. For all the other edges, connect with probability $1/2$. 

\paragraph{Hidden Clique Hypothesis for Localization ($\sf HC_{l}$)} 
Consider the random instance of hidden clique model $\mc{G}(N,\kappa)$.
For any sequence $\kappa(N)$ such that $\kappa(N) \leq N^\beta$ for some $0 < \beta < 1/2$, there is no randomized polynomial time algorithm that can find the planted clique with probability tending to $1$ as $N \rightarrow \infty$. Mathematically, define the randomized polynomial time algorithm class $\sf PolyAlg$ as the class of algorithms $\mathcal{A}$ that satisfies
$$
\varlimsup\limits_{N,\kappa(N) \rightarrow \infty} \sup_{\mathcal{A} \in {\sf PolyAlg}} \mathbb{E}_{\sf Clique} \mathbb{P}_{\mc{G}(N,\kappa)|{\sf Clique}}\left( {\sf runtime~of}~\mathcal{A}~ {\sf not~polynomial~in}~N \right) = 0.
$$
Then
$$
\varliminf\limits_{N,\kappa(N) \rightarrow \infty} \inf_{\mathcal{A} \in {\sf PolyAlg}}  \mathbb{E}_{\sf Clique} \mathbb{P}_{\mc{G}(N,\kappa)|{\sf Clique}}\left( {\sf clique~set~returned~by~}\mathcal{A}~{\sf not~correct} \right) > 0,
$$
where $\mathbb{P}_{\mc{G}(N,\kappa)|{\sf Clique}}$ is the (possibly more detailed due to randomness of algorithm) $\sigma$-field conditioned on the clique location and $\mathbb{E}_{\sf Clique}$ is with respect to uniform distribution over all possible clique locations.

\paragraph{Hidden Clique Hypothesis for Detection ($\sf HC_{d}$)}
Consider the hidden clique model $\mc{G}(N,\kappa)$.
For any sequence of $\kappa(N)$ such that $\kappa(N) \leq N^{\beta}$ for some $0 < \beta < 1/2 $, there is no randomized polynomial time algorithm that can distinguish between 
\begin{align*}
	{\sf H_0}:  ~~\mathcal{P}_{\sf ER} \quad  {\sf v.s.} \quad
	{\sf H_\alpha}:  ~~\mathcal{P}_{\sf HC}
\end{align*}
with probability going to $1$ as $N \rightarrow \infty$. Here $\mathcal{P}_{\sf ER}$ is the Erd\H{o}s-R\'{e}nyi model, while $\mathcal{P}_{\sf HC}$ is the hidden clique model with uniform distribution on all the possible locations of the clique. More precisely,
$$
\varliminf\limits_{N,\kappa(N) \rightarrow \infty} \inf_{\mathcal{A} \in {\sf PolyAlg}}  \mathbb{E}_{\sf Clique} \mathbb{P}_{\mc{G}(N,\kappa)|{\sf Clique}}\left( {\sf detection~decision~returned~by~}\mathcal{A}~{\sf wrong} \right) > 0,
$$
where $\mathbb{P}_{\mc{G}(N,\kappa)|{\sf Clique}}$ and $\mathbb{E}_{\sf Clique}$ are the same as defined in ${\sf HC_l}$.

The hidden clique hypothesis has been used recently by several authors to claim computational intractability of certain statistical problems. In particular, \cite{berthet2013computational,ma2013computational} assumed the hypothesis ${\sf HC_{d}}$ and \cite{wang2014statistical} used $\sf HC_{l}$. 
Localization is harder than detection, in the sense that if an algorithm $\mathcal{A}$ solves the localization problem with high probability, it also correctly solves the detection problem. Assuming that no polynomial time algorithm can solve the detection problem implies impossibility results in localization as well. In plain language, $\sf HC_{l}$ is a milder hypothesis than $\sf HC_{d}$. 

We will provide two computational lower bound results, one for localization and the other for detection, in Theorems \ref{CLL.Thm} and \ref{CLD.Thm}. The latter one will be deferred to Appendix \ref{sec:complow-det} to contrast the difference of constructions between localization and detection. The detection computational lower bound was first proved in \cite{ma2013computational}. For the localization computational lower bound, to the best of our knowledge, there is no proof in the literature. Theorem \ref{CLL.Thm} ensures the upper bound in Lemma \ref{lma:sp-alg} being sharp. 

\begin{theorem}[Computational Lower Bound for Localization]
	\label{CLL.Thm}
	Consider the submatrix model \eqref{submat.Mat.Frm} with parameter tuple $(m = n, k_m \asymp k_n \asymp n^{\alpha}, \lambda/\sigma = n^{-\beta})$, where $\frac{1}{2}<\alpha<1,~ \beta>0$. Under the computational assumption $\sf HC_{l}$, if 
	\begin{align*}
		\frac{\lambda}{\sigma} \precsim \sqrt{\frac{m+n}{k_m k_n}} ~~ \Rightarrow ~~\beta > \alpha -\frac{1}{2},
	\end{align*}
	it is not possible to localize the true support of the submatrix with probability going to $1$ within polynomial time.
\end{theorem}

Our algorithmic reduction for localization relies on a \emph{bootstrapping} idea based on the matrix structure and a cleaning-up procedure introduced in Lemma \ref{lma:neig} given in Section \ref {sec:pf}. These two key ideas offer new insights in addition to the usual computational lower bound arguments. Bootstrapping introduces an additional randomness on top of the randomness in the hidden clique. Careful examination of these two $\sigma$-fields allows us to write the resulting object into mixture of submatrix models. For submatrix localization we need to transform back the submatrix support to the original hidden clique support exactly, with high probability. In plain language, even though we lose track of the exact location of the support when reducing the hidden clique to submatrix model, we can still recover the exact location of the hidden clique with high probability. For technical details of the proof, please refer to Section \ref{sec:pf}.

\subsection{Adaptive Spectral Algorithm and Computational Upper Bound}
\label{subsec:alg}

In this section, we introduce linear time algorithm that solves the submatrix localization problem above the computational boundary $\sf SNR_c$. Our proposed localization Algorithms \ref{SVD1.Alg} and \ref{SVDThre.Alg} is motivated by the spectral algorithm in random graphs \citep{mcsherry2001spectral,ng2002spectral}.

\medskip
\begin{algorithm}[H]
\KwIn{$X \in \mbb{R}^{m \times n}$ the data matrix. }
\KwOut{A subset of the row indexes $\hat{R}_m$ and a subset of column indexes $\hat{C}_n$ as the localization sets of the submatrix.}
    1. Compute top left and top right singular vectors $U_{\cdot 1}$ and $V_{\cdot 1}$, respectively (these correspond to the SVD $X = U \Sigma V^T$)\;
    2. To compute $\hat{C}_n$, calculate the inner products $U_{\cdot 1}^T X_{\cdot j}, 1\leq j \leq n$. These values form two clusters. Similarly, for the $\hat{R}_m$, calculate $X_{i \cdot} V_{\cdot 1}, 1\leq i \leq m$ and obtain two separated clusters. A simple thresholding procedure returns the subsets $\hat{C}_n$ and $\hat{R}_m$.
\caption{Vanilla Spectral Projection Algorithm for Dense Regime} 
\label{SVD1.Alg}
\end{algorithm}

\medskip
The proposed algorithm has several advantages over the localization algorithms that appeared in literature. First, it is a linear time algorithm (that is, $\Theta(mn)$ time complexity). The top singular vectors can be evaluated using fast iterative power methods, which is efficient both in terms of space and time. Secondly, this algorithm does not require the prior knowledge of $k_m$ and $k_n$ and automatically adapts to the true submatrix size.


Lemma~\ref{lma:sp-alg} below justifies the effectiveness of the spectral algorithm. 

\begin{lemma}[Guarantee for Spectral Algorithm]
	\label{lma:sp-alg}
	Consider the submatrix model \eqref{submat.Mat.Frm}, Algorithm \ref{SVD1.Alg} and assume $\min\{k_m,k_n\} \succsim \max\{m^{1/2},n^{1/2}\}$. There exist a universal $C>0$ such that when 
	$$
	\frac{\lambda}{\sigma} \geq C \cdot \left(  \sqrt{\frac{m \vee n}{k_m k_n}} + \sqrt{\frac{\log n}{k_m} \vee \frac{\log m}{k_n}} \right),
	$$
	the spectral method succeeds in the sense that $\hat{R}_m = R_m, \hat{C}_n = C_n$ with probability at least $1 - m^{-c} - n^{-c} - 2 \exp\left(-c(m+n)\right)$.
\end{lemma}


\subsection{Dense Regime}
\label{subsec:dense}
We are now ready to state the SNR boundary for polynomial-time algorithms (under an appropriate computational assumption), thus excluding the exhaustive search procedure. The results hold under the dense regime when $k \succsim n^{1/2}$.

\begin{theorem}[Computational Boundary for Dense Regime]
	\label{thm:comp-bdr}
	Consider the submatrix model \eqref{submat.Mat.Frm} and assume $\min\{k_m,k_n\} \succsim \max\{m^{1/2},n^{1/2}\}$. There exists a critical rate 
	$$
	{\sf SNR_c} \asymp  \sqrt{\frac{m \vee n}{k_m k_n}} + \sqrt{\frac{\log n}{k_m} \vee \frac{\log m}{k_n}}
	$$
	for the signal to noise ratio $\sf SNR_c$ such that for $\lambda/\sigma \succsim {\sf SNR_c}$, both the adaptive linear time Algorithm \ref{SVD1.Alg} and the robust polynomial time Algorithm \ref{CR1.Alg} will succeed in submatrix localization, i.e., $\hat{R}_m = R_m, \hat{C}_n = C_n$, with high probability. For $\lambda/\sigma \precsim {\sf SNR_c}$, there is no polynomial time algorithm that will work under the hidden clique hypothesis $\sf HC_{l}$. 
\end{theorem}

The proof of the above theorem is based on the theoretical justification of the spectral Algorithm \ref{SVD1.Alg} and convex relaxation Algorithm \ref{CR1.Alg}, and the new computational lower bound result for localization in Theorem \ref{CLL.Thm}. 
We remark that the analyses can be extended to multiple, even growing number of submatrices case.
We postpone a proof of this fact to Section \ref{sec:ext} for simplicity and focus on the case of a single submatrix.

\subsection{Sparse Regime}
\label{subsec:sparse}
Under the sparse regime when $k \precsim n^{1/2}$, a naive plug-in of Lemma~\ref{lma:sp-alg} requires the ${\sf SNR_c}$ to be larger than $\Theta(n^{1/2}/k) \succsim \sqrt{\log n}$, which implies the vanilla spectral Algorithm~\ref{SVD1.Alg} is outperformed by simple entrywise thresholding. However, a modified version with entrywise soft-thresholding as a preprocessing de-noising step turns out to provide near optimal performance in the sparse regime. Before we introduce the formal algorithm, let us define the soft-thresholding function at level $t$ to be
\begin{align}
	\label{eq:soft-thres}
\eta_{t} (y) = \sign(y) (|y| - t)_+.
\end{align}
Soft-thresholding as a de-noising step achieving optimal bias-and-variance trade-off has been widely understood in the wavelet literature, for example, see \citet{donoho1998minimax}.

Now we are ready to state the following de-noised spectral Algorithm~\ref{SVDThre.Alg} to localize the submatrix under the sparse regime when $k \precsim n^{1/2}$. 

\medskip
\begin{algorithm}[H]
\KwIn{$X \in \mbb{R}^{m \times n}$ the data matrix, a thresholding level $t = \Theta(\sigma \sqrt{\log \frac{m \vee n}{k_m k_n}})$. }
\KwOut{A subset of the row indexes $\hat{R}_m$ and a subset of column indexes $\hat{C}_n$ as the localization sets of the submatrix.}
	1. Soft-threshold each entry of the matrix $X$ at level $t$, denote the resulting matrix as $\eta_t(X)$ \;
    2. Compute top left and top right singular vectors $U_{\cdot 1}$ and $V_{\cdot 1}$ of matrix $\eta_t(X)$, respectively (these correspond to the SVD $\eta_t(X) = U \Sigma V^T$)\;
    3. To compute $\hat{C}_n$, calculate the inner products $U_{\cdot 1}^T \cdot \eta_t(X_{\cdot j}), 1\leq j \leq n$. These values form two clusters. Similarly, for the $\hat{R}_m$, calculate $\eta_t(X_{i \cdot}) \cdot V_{\cdot 1}, 1\leq i \leq m$ and obtain two separated clusters. A simple thresholding procedure returns the subsets $\hat{C}_n$ and $\hat{R}_m$.
\caption{De-noised Spectral Algorithm for Sparse Regime} 
\label{SVDThre.Alg}
\end{algorithm}

\medskip

Lemma~\ref{lma:thres-sp-alg} below provides the theoretical guarantee for the above algorithm when $k \precsim n^{1/2}$. 

\begin{lemma}[Guarantee for De-noised Spectral Algorithm]
	\label{lma:thres-sp-alg}
	Consider the submatrix model \eqref{submat.Mat.Frm}, soft-thresholded spectral Algorithm \ref{SVDThre.Alg} with thresholded level $\sigma t$, and assume $\min \{k_m,k_n\} \precsim \max \{m^{1/2},n^{1/2}\}$. There exist a universal $C>0$ such that when 
	$$
	\frac{\lambda}{\sigma} \geq C \cdot \left(  \left[ \sqrt{\frac{m \vee n}{k_m k_n}} + \sqrt{\frac{\log n}{k_m} \vee \frac{\log m}{k_n}} \right] \cdot e^{-t^2/2} + t \right),
	$$
	the spectral method succeeds in the sense that $\hat{R}_m = R_m, \hat{C}_n = C_n$ with probability at least $1 - m^{-c} - n^{-c} - 2 \exp\left(-c(m+n)\right)$. Further if we choose $t = \Theta(\sigma \sqrt{\log \frac{m \vee n}{k_m k_n}})$ as the optimal thresholding level, we have de-noised spectral algorithm works when
	$$
	\frac{\lambda}{\sigma} \succsim \sqrt{\log \frac{m\vee n}{k_m k_n}}.
	$$
\end{lemma}

Combining the hidden clique hypothesis $\sf HC_{l}$ together with Lemma~\ref{lma:thres-sp-alg}, we have the following theorem holds under the sparse regime when $k \precsim n^{1/2}$.

\begin{theorem}[Computational Boundary for Sparse Regime]
	\label{thm:comp-bdr-sparse}
	Consider the submatrix model \eqref{submat.Mat.Frm} and assume $\max\{k_m,k_n\} \precsim \min\{m^{1/2},n^{1/2}\}$. There exists a critical rate for the signal to noise ratio $\sf SNR_c$ between
	$$
	1 \precsim {\sf SNR_c} \precsim \sqrt{\log \frac{m\vee n}{k_m k_n}}
	$$
	such that for $\lambda/\sigma \succsim \sqrt{\log \frac{m\vee n}{k_m k_n}}$, the linear time Algorithm \ref{SVDThre.Alg} will succeed in submatrix localization, i.e., $\hat{R}_m = R_m, \hat{C}_n = C_n$, with high probability. For $\lambda/\sigma \precsim 1$, there is no polynomial time algorithm that will work under the hidden clique hypothesis $\sf HC_{l}$.  
\end{theorem}

\begin{remark}{\rm
	The upper bound achieved by the de-noised spectral Algorithm~\ref{SVDThre.Alg} is optimal in the two boundary cases: $k = 1$ and $k \asymp n^{1/2}$. When $k = 1$, both the information theoretic and computational boundary meet at $\sqrt{\log n}$. When $k \asymp n^{1/2}$, the computational lower bound and upper bound match in Theorem~\ref{thm:comp-bdr-sparse}, thus suggesting the near optimality of Algorithm~\ref{SVDThre.Alg} within the polynomial time algorithm class. The potential logarithmic gap is due to the crudeness of the hidden clique hypothesis. Precisely, for $k = 2$, hidden clique is not only hard for $G(n,p)$ with $p=1/2$, but also hard for $G(n,p)$ with $p = 1/\log n$. Similarly for $k = n^{\alpha}, \alpha < 1/2$, hidden clique is not only hard for $G(n,p)$ with $p=1/2$, but also for some $0<p<1/2$. 
	}
\end{remark}

\subsection{Extension to Growing Number of Submatrices}
\label{sec:ext}

The computational boundaries established in the previous sections for a single submatrix can be extended to non-overlapping multiple submatrices model \eqref{multmat.Mat.Frm}. The non-overlapping assumption corresponds to that for any $1\leq s \neq t \leq r$, $R_s \cap R_t = \emptyset$ and $C_s \cap C_t = \emptyset$. 
The Algorithm \ref{SVD2.Alg} below is an extension of the spectral projection Algorithm \ref{SVD1.Alg} to address the multiple submatrices localization problem. 

\medskip
\begin{algorithm}[H]
\KwIn{$X \in \mbb{R}^{m \times n}$ the data matrix. A pre-specified number of submatrices $r$.}
\KwOut{A subset of the row indexes $\{\hat{R}_m^s, 1\leq s\leq r \}$ and a subset of column indexes $\{ \hat{C}_n^s, 1\leq s\leq r\}$ as the localization of the submatrices.}
    1. Calculate top $r$ left and right singular vectors in the SVD $X = U \Sigma V^T$. Denote these vectors as $U_r \in \mathbb{R}^{m \times r}$ and $V_r \in \mathbb{R}^{n \times r}$, respectively\;
    2. For the $\hat{C}_n^s, 1 \leq s \leq r$, calculate the projection $U_r (U_r^T U_r)^{-1} U_r^T X_{\cdot j}, 1\leq j \leq n$, run $k$-means clustering algorithm (with $k = r+1$) for these $n$ vectors in $\mbb{R}^m$. For the $\hat{R}_m^s, 1 \leq s \leq r$, calculate $V_r (V_r^T V_r)^{-1} V_r^T X_{i \cdot}^T, 1 \leq i \leq m$, run $k$-means clustering algorithm (with $k = r+1$) for these $m$ vectors in $\mbb{R}^n$ (while the effective dimension is $\mbb{R}^r$).
\caption{Spectral Algorithm for Multiple Submatrices} 
\label{SVD2.Alg}
\end{algorithm}

\medskip
We emphasize that the following Proposition \ref{lma:sp-ext} holds even when the number of submatrices $r$ grows with $m,n$.

\begin{lemma}[Spectral Algorithm for Non-overlapping Submatrices Case]
	\label{lma:sp-ext}
	Consider the non-overlapping multiple submatrices model \eqref{multmat.Mat.Frm} and Algorithm \ref{SVD2.Alg}. Assume $$k_s^{(m)} \asymp k_m, k_s^{(n)} \asymp k_n, \lambda_s \asymp \lambda$$ for all $1\leq s \leq r$ and $\min\{k_m,k_n\} \succsim \max\{m^{1/2},n^{1/2}\}$. There exist a universal $C>0$ such that when 
	\begin{align}
		\label{eq:multi-spectral}
	\frac{\lambda}{\sigma} \geq C \cdot \left( \sqrt{\frac{r}{k_m \wedge k_n}} +  \sqrt{\frac{\log n}{k_m }}\vee \sqrt{\frac{\log m}{k_n}} +  \sqrt{\frac{m\vee n}{k_m k_n}} \right),
	\end{align}
	the spectral method succeeds in the sense that $\hat{R}_m^{(s)} = R_m^{(s)}, \hat{C}_n^{(s)} = C_n^{(s)}, 1\leq s\leq r$ with probability at least $1 - m^{-c} - n^{-c} - 2 \exp\left(-c(m+n)\right)$.
\end{lemma}
\begin{remark}{\rm
	Under the non-overlapping assumption, $r k_m \precsim m, ~r k_n \precsim n$ hold in most cases. Thus the first term in Equation \eqref{eq:multi-spectral} is dominated by the latter two terms. Thus a growing number $r$ does not affect the bound in Equation \eqref{eq:multi-spectral} as long as the non-overlapping assumption holds.}
\end{remark}

\section{Statistical Boundary}
\label{sec:stat-bdr}
In this section we study the statistical boundary. As mentioned in the introduction, in the Gaussian noise setting, the statistical boundary for a single submatrix localization has been established in \cite{butucea2013sharp}. In this section, we generalize to localization of a growing number of submatrices, as well as sub-Gaussian noise, at the expense of having non-exact constants for the threshold.

\subsection{Information Theoretic Bound}

We begin with the information theoretic lower bound for the localization accuracy.

\begin{lemma}[Information Theoretic Lower Bound]
	\label{lma:info-low}
	Consider the submatrix model \eqref{submat.Mat.Frm} with Gaussian noise $Z_{ij} \sim \mathcal{N}(0, \sigma^2)$. For any fixed $0<\alpha<1$, there exist a universal constant $C_\alpha$ such that if
	\begin{align}
		\frac{\lambda}{\sigma} \leq C_\alpha \cdot \sqrt{\frac{\log (m/k_m)}{k_n}+ \frac{\log (n/k_n)}{k_m}},
	\end{align}
	any algorithm $\mathcal{A}$ will fail to localize the submatrix with probability at least $1-\alpha - \frac{\log 2}{k_m \log(m/k_m) + k_n \log (n/k_n)}$ in the following minimax sense:
	$$
	\inf_{\mathcal{A} \in {\sf AllAlg}}~ \sup_{M \in \mathcal{M}}~~~ \mathbb{P}\left( \hat{R}^{\mathcal{A}}_m \neq R_m ~~ \text{or}  ~~ \hat{C}^{\mathcal{A}}_n \neq C_n \right) > 1-\alpha - \frac{\log 2}{k_m \log(m/k_m) + k_n \log (n/k_n)}.
	$$
\end{lemma}

\subsection{Combinatorial Search for Growing Number of Submatrices}


Combinatorial search over all submatrices of size $k_m \times k_n$ finds the location with the strongest aggregate signal and is statistically optimal \citep{butucea2013sharp,butucea2013detection}. Unfortunately, it requires computational complexity $\Theta\left(\binom{m}{k_m}+\binom{n}{k_n}\right)$, which is exponential in $k_m, k_n$. 
The search Algorithm~\ref{SS.Alg} was introduced and analyzed under the Gaussian setting for a single submatrix in \cite{butucea2013detection}, which can be used iteratively to solve multiple submatrices localization.

\medskip
\begin{algorithm}[H]
\KwIn{$X \in \mbb{R}^{m \times n}$ the data matrix.}
\KwOut{A subset of the row indexes $\hat{R}_m$ and a subset of column indexes $\hat{C}_n$ as the localization of the submatrix.}
    For all index subsets $I \times J$ with $|I| = k_m$ and $|J| = k_n$, calculate the sum of the entries in the submatrix $X_{I J}$. Report the index subset $\hat{R}_m \times \hat{C}_n$ with the largest sum.
\caption{Combinatorial Search Algorithm}
\label{SS.Alg}
\end{algorithm}

\medskip
For the case of multiple submatrices, the submatrices can be extracted with the largest sum in a greedy fashion.

Lemma~\ref{lma:stat-alg} below provides a theoretical guarantee for Algorithm \ref{SS.Alg} to achieve the information theoretic lower bound.
\begin{lemma}[Guarantee for Search Algorithm]
	\label{lma:stat-alg}
	Consider the non-overlapping multiple submatrices model \eqref{multmat.Mat.Frm} and iterative application of Algorithm \ref{SS.Alg} in a greedy fashion for $r$ times. Assume $$k_s^{(m)} \asymp k_m, k_s^{(n)} \asymp k_n, \lambda_s \asymp \lambda$$ for all $1\leq s \leq r$ and $\max\{k_m,k_n\} \precsim \min\{m,n\}$. There exists a universal constant $C>0$ such that if
	$$
	\frac{\lambda}{\sigma} \geq C \cdot \sqrt{\frac{ \log(em/k_m) }{k_n} + \frac{ \log (en/k_n)}{k_m}},
	$$
	then Algorithm \ref{SS.Alg} will succeed in returning the correct location of the submatrix with probability at least $1 - \frac{2k_m k_n}{mn}$.
\end{lemma}



To complete Theorem \ref{thm:main}, we include the following Theorem~\ref{thm:stat-bdr} capturing the statistical boundary. It is proved by exhibiting the information-theoretic lower bound Lemma \ref{lma:info-low} and analyzing Algorithm \ref{SS.Alg}. 
\begin{theorem}[Statistical Boundary]
	\label{thm:stat-bdr}
	Consider the submatrix model \eqref{submat.Mat.Frm}. There exists a critical rate 
	$$
	{\sf SNR}_s \asymp  \sqrt{\frac{\log n}{k_m} \vee \frac{\log m}{k_n}}
	$$ 
	for the signal to noise ratio, such that for any problem with $\lambda/\sigma \succsim {\sf SNR_s}$, the statistical search Algorithm \ref{SS.Alg} will succeed in submatrix localization, i.e., $\hat{R}_m = R_m, \hat{C}_n = C_n$, with high probability. On the other hand, if $\lambda/\sigma \precsim {\sf SNR_s}$, no algorithm will work (in the minimax sense) with probability tending to $1$. 
\end{theorem}

\section{Discussion}
\label{sec:dis}

In this paper we established the computational and statistical boundaries for submatrix localization in the setting of a growing number of submatrices with subgaussian noise. The primary goals are to demonstrate the intrinsic gap between what is statistical possible and what is computationally feasible and to contrast the interplay between computational efficiency and statistical accuracy for localization with that for detection.

\paragraph{Submatrix Localization v.s. Detection}

As pointed out in Section \ref{sec:contribution}, for any $k = n^{\alpha}, 0<\alpha<1$, there is an intrinsic SNR gap between computational and statistical boundaries for submatrix localization. Unlike the submatrix detection problem where for the regime $2/3<\alpha<1$, there is no gap between what is computationally possible and what is statistical possible. The inevitable gap in submatrix localization is due to the combinatorial structure of the problem. This phenomenon is also seen in some network related problems, for instance, stochastic block models with a growing number of communities. Compared to the submatrix detection problem, the algorithm to solve the localization problem is more complicated and the techniques required for the analysis are much more involved.

\paragraph{Detection for Growing Number of Submatrices}
The current paper solves localization of a growing number of submatrices. In comparison, for detection, the only known results are for the case of a single submatrix as considered in \cite{butucea2013detection} for the statistical boundary and in  \cite{ma2013computational} for the computational boundary. The detection problem in the setting of a growing number of submatrices is of significant interest. In particular, it is interesting to understand the computational and statistical trade-offs in such a setting. This will need further investigation.

\paragraph{Estimation of the Noise Level $\sigma$}
Although Algorithms \ref{SVD1.Alg} and \ref{SVD2.Alg} do not require the noise level $\sigma$ as an input, Algorithm \ref{SVDThre.Alg} does require the knowledge of $\sigma$.
The noise level $\sigma$ can be estimated robustly. In the Gaussian case,  a simple robust estimator of $\sigma$ is the following median absolute deviation (MAD) estimator due to the fact that $M$ is sparse:
\begin{align*}
\hat{\sigma} &= {\rm median}_{ij} |X_{ij} - {\rm median}_{ij}(X_{ij})|/\Phi^{-1}(0.75)\\
& \approx 1.4826 \times {\rm median}_{ij}|X_{ij}- {\rm median}_{ij}(X_{ij})|. 
\end{align*}

\section{Proofs}
\label{sec:pf}

We prove in this section the main results given in the paper. We first  collect and prove a few important technical lemmas that will be used in the proofs of the main results.

\subsection{Prerequisite Lemmas}
We start with two Lemmas \ref{lma:stewart} and \ref{lma:dkw} that due to perturbation theory. 
\begin{lemma}[\citet{stewart1990matrix} Theorem 4.1]
	\label{lma:stewart}
Suppose that $\tilde{A} = A + E$, all of which are matrices of the same size, and we have the following singular value decomposition
\begin{equation}
[U_1, U_2, U_3]^T A [V_1, V_2] = \left[
  \begin{array}{cc}
  \Sigma_1 & 0\\
     0     & \Sigma_2 \\
     0     & 0
\end{array} \right]
\end{equation}
and 
\begin{equation}
[\tilde{U}_1, \tilde{U}_2,\tilde{U}_3]^T \tilde{A} [\tilde{V}_1, \tilde{V}_2] = \left[
\begin{array}{cc}
\tilde{\Sigma}_1& 0\\
0&\tilde{\Sigma}_2 \\
0&0
\end{array} \right].
\end{equation}
Let $\Phi$ be the matrix of canonical angles between $\mathcal{R}(U_1)$ and $\mathcal{R}(\tilde{U}_1)$, and let $\Theta$ be the matrix of canonical angles between $\mathcal{R}(V_1)$ and $\mathcal{R}(\tilde{V}_1)$ (here $\mathcal{R}$ denotes the linear space). Define
\begin{align}
R &= A \tilde{V}_1 - \tilde{U}_1 \tilde{\Sigma}_1 \\
S &= A^T \tilde{U}_1 - \tilde{V}_1 \tilde{\Sigma}_1
\end{align}
Then suppose there is a number $\delta>0$ such that
$$
\min |\sigma(\tilde{\Sigma}_1) - \sigma(\Sigma_2)| \geq \delta \quad \text{and} \quad \min \sigma(\tilde{\Sigma}_1) \geq \delta.
$$
Then
$$
\sqrt{\| \sin \Phi \|_F^2 + \| \sin \Theta \|_F^2} \leq  \frac{\sqrt{\|R \|_F^2 + \| S \|_F^2}}{\delta}.
$$
Further, suppose there are numbers $\alpha,\delta$ such that
$$
\min \sigma(\tilde{\Sigma}_1) \geq \delta + \alpha \quad \text{and} \quad \max \sigma(\Sigma_2) \leq \alpha,
$$
then for $2$-norm, or any unitarily invariant norm, we have
$$
\max \left\{ \| \sin \Phi \|_2, \| \sin \Theta \|_2 \right\} \leq  \frac{\max\{\|R \|_2, \| S \|_2\}}{\delta}.
$$
\end{lemma}

Let us use the above version of the perturbation bound to derive a lemma that is particularly useful in our case. Simple algebra tells us that
\begin{equation}
\tilde{A} [\tilde{V}_1, \tilde{V}_2] =  [\tilde{U}_1, \tilde{U}_2,\tilde{U}_3]  \left[
\begin{array}{cc}
\tilde{\Sigma}_1& 0\\
0&\tilde{\Sigma}_2 \\
0&0
\end{array} \right] = [\tilde{U}_1 \tilde{\Sigma}_1,\tilde{U}_2 \tilde{\Sigma}_2].
\end{equation}

\begin{align}
A[\tilde{V}_1, \tilde{V}_2]               & = [A \tilde{V}_1, A\tilde{V}_2]  \\
(\tilde{A} - A)[\tilde{V}_1, \tilde{V}_2] & = [\tilde{U}_1 \tilde{\Sigma}_1 - A \tilde{V}_1,\tilde{U}_2 \tilde{\Sigma}_2-A\tilde{V}_2] \\
\| \tilde{A} - A \|_F^2                   & = {\sf Tr}\left((\tilde{A} - A)[\tilde{V}_1, \tilde{V}_2] [\tilde{V}_1, \tilde{V}_2]^T (\tilde{A} - A)^T \right) \\
                                          & = \| R \|_F^2 + \|\tilde{U}_2 \tilde{\Sigma}_2-A\tilde{V}_2 \|_F^2 \geq \| R\|_F^2.
\end{align}
Similarly, we have 
\begin{align}
(\tilde{A} - A)^T [\tilde{U}_1, \tilde{U}_2, \tilde{U}_3] & = [\tilde{V}_1 \tilde{\Sigma}_1 - A^T \tilde{U}_1 ,\tilde{V}_2 \tilde{\Sigma}_2 - A^T \tilde{U}_2, -A^T \tilde{U}_3 ] \\
\| \tilde{A} - A \|_F^2                                   & = {\sf Tr}\left( (\tilde{A} - A)^T [\tilde{U}_1, \tilde{U}_2, \tilde{U}_3] [\tilde{U}_1, \tilde{U}_2, \tilde{U}_3]^T (\tilde{A} - A)\right) \\
& = \| S \|_F^2 + \|\tilde{V}_2 \tilde{\Sigma}_2 - A^T \tilde{U}_2 \|_F^2 + \| A^T \tilde{U}_3 \|_F^2 \geq \| S \|_F^2.
\end{align}
Thus, it holds that
$$
\| \tilde{A} - A \|_F \geq \max (\| R \|_F, \| S \|_F)
$$
and similarly we have (since the operator norm of a whole matrix is larger than that of the submatrix)
$$
\| \tilde{A} - A \|_2 \geq \max (\| R \|_2, \| S \|_2).
$$

Thus the following version of the Wedin's Theorem holds.
\begin{lemma}[Davis-Kahan-Wedin's Type Perturbation Bound]
\label{lma:dkw}
It holds that
$$
\sqrt{\| \sin \Phi \|_F^2 + \| \sin \Theta \|_F^2} \leq  \frac{\sqrt{2} \| E\|_F}{\delta}
$$
and also the following holds for $2$-norm (or any unitary invariant norm)
$$
\max \left\{ \| \sin \Phi \|_2, \| \sin \Theta \|_2 \right\} \leq  \frac{\| E\|_2}{\delta}.
$$
\end{lemma}

We will then introduce some concentration inequalities. Lemmas \ref{lma:rmt} and \ref{Proj.Lma} are concentration of measure results from random matrix theory.
\begin{lemma}[\cite{vershynin2010introduction}, Theorem 39]
	\label{lma:rmt}
	Let $Z \in \mathbb{R}^{m \times n}$ be a matrix whose rows $Z_{i \cdot}$ are independent sub-Gaussian isotropic random vectors in $\mathbb{R}^n$ with parameter $\sigma$. Then for every $t \geq 0$, with probability at least $1 - 2 \exp(-c t^2)$ one has
	$$
	\| Z \|_{2} \leq \sigma (\sqrt{m} + C \sqrt{n} + t)
	$$
	where $C, c>0$ are some universal constants.
\end{lemma}

\begin{lemma}[\cite{hsu2012tail}, Projection Lemma]
	\label{Proj.Lma}
	Assume $Z \in \mathbb{R}^n$ is an isotropic sub-Gaussian vector with i.i.d. entries and parameter $\sigma$. $\mathcal{P}$ is a projection operator to a subspace of dimension $r$, then we have the following concentration inequality
	\begin{align*}
		\mathbb{P}(\| \mathcal{P} Z \|_{\ell_2}^2  \geq  \sigma^2 (r + 2 \sqrt{r t} + 2 t )  ) \leq \exp(-c t),
	\end{align*}
	where $c>0$ is a universal constant. 
\end{lemma}
The proof of this lemma is a simple application of Theorem 2.1 in\cite{hsu2012tail} for the case that $\mathcal{P}$ is a rank $r$ positive semidefinite projection matrix.

The following two are standard Chernoff-type bounds for bounded random variables.

	 \begin{lemma}[\cite{hoeffding1963probability}, Hoeffding's Inequality]
		 Let $X_i, 1\leq i\leq n$ be independent random variables. Assume $a_i \leq X_i \leq b_i, 1\leq i\leq n$. Then for $S_n = \sum_{i=1}^n X_i$
		 \begin{align}
			 \mathbb{P}\left( |S_n - \mathbb{E} S_n|>u \right) \leq 2 \exp\left( - \frac{2u^2}{\sum_{i=1}^n (b_i - a_i)^2}\right).
		 \end{align}
	 \end{lemma}
	 \begin{lemma}[\cite{bennett1962probability}, Bernstein's Inequality]
		 Let $X_i,1\leq i\leq n$ be independent zero-mean random variables. Suppose $|X_i|\leq M, 1\leq i \leq n$. Then 
		 \begin{align}
			 \mathbb{P}\left( \sum_{i=1}^n X_i > u \right) \leq \exp \left( - \frac{u^2/2}{\sum_{i=1}^n \mathbb{E} X_i^2 + Mu/3}\right).
		 \end{align}
	 \end{lemma}
	
	We will end this section stating the Fano's information inequality, which plays a key role in many information theoretic lower bounds.
 	\begin{lemma}[\cite{tsybakov2009introduction} Corollary 2.6]
 		\label{lma:fano}
 	Let $\mc{P}_0, \mc{P}_1,\ldots,\mc{P}_M$ be probability measures on the same probability space $(\Theta,\mathcal{F})$, $M \geq 2$. If for some $0<\alpha<1$
 	\begin{align}
 		\label{Fano.Cnd}
 	\frac{1}{M+1} \sum_{i=0}^M \KL(\mc{P}_i || \bar{\mc{P}}) \leq \alpha  \cdot \log M
 	\end{align}
 	where
 	$$
 	\bar{\mc{P}} = \frac{1}{M+1} \sum_{i=0}^M \mc{P}_i.
 	$$ 
 	Then 
 	\begin{align}
 	p_{e,M} \geq \bar{p}_{e,M} \geq \frac{\log(M+1)-\log 2}{\log M} - \alpha
 	\end{align}
 	where $p_{e, M}$ is the minimax error for the multiple testing problem.
 	\end{lemma}

\subsection{Main Proofs}

\begin{proof}[Proof of Lemma \ref{lma:sp-alg}]
Recall the matrix form of the submatrix model, with the SVD decomposition of the mean signal matrix $M$
$$
X = \lambda \sqrt{k_m k_n} U V^T + Z.
$$
The largest singular value of $\lambda U V^T$ is $\lambda \sqrt{k_m k_n}$, and all the other singular values are $0$s. Davis-Kahan-Wedin's perturbation bound tells us how close the singular space of $X$ is to the singular space of $M$.
Let us apply the derived Lemma~\ref{lma:dkw} to $X = \lambda \sqrt{k_m k_n} U V^T +Z$. Denote the top left and right singular vector of $X$ as $\tilde{U}$ and $\tilde{V}$. One can see that $\mathbb{E} \| Z \|_2 \asymp \sigma(\sqrt{m}+\sqrt{n})$ under very mild finite fourth moment conditions through a result in \citep{latala2005some}. Lemma \ref{lma:rmt} provides a more explicit probabilisitic bound for the concentration of the largest singular value of i.i.d sub-Gaussian random matrix. 
Because the rows $Z_{i \cdot}$ are sampled from product measure of mean zero sub-Gaussians, they naturally satisfy the isotropic condition. Hence, with probability at least $1 - 2 \exp\left( - c (m+n)\right)$, via Lemma \ref{lma:rmt}, we reach
\begin{align}
	\label{op.eq}
\| Z \|_2 \leq C \cdot \sigma(\sqrt{m}+\sqrt{n}).
\end{align}

Using Weyl's interlacing inequality, we have
$$
| \sigma_i(X)  - \sigma_i(M) | \leq \| Z \|_2
$$
and thus
$$
\sigma_1(X) \geq \lambda \sqrt{k_m k_n}  - \| Z \|_{2}
$$
$$
\sigma_2(X) \leq \| Z \|_{2}.
$$
Applying Lemma \ref{lma:dkw}, we have
$$
\max \left\{ |\sin \angle (U, \tilde{U}) |, |\sin \angle (V, \tilde{V})| \right\} \leq \frac{C \sigma (\sqrt{m}+\sqrt{n})}{\lambda \sqrt{k_m k_n} -  C \sigma (\sqrt{m}+ \sqrt{n})} \asymp \frac{\sigma (\sqrt{m}+\sqrt{n})}{\lambda \sqrt{k_m k_n}}.
$$
In addition
$$
\| U - \tilde{U} \|_{\ell_2}  = \sqrt{2 - 2 \cos \angle (U, \tilde{U}) } = 2 |\sin \frac{1}{2}\angle(U, \tilde{U} ) |,
$$
which means
$$
\max \left\{ \| U - \tilde{U} \|_{\ell_2}, \| V - \tilde{V}\|_{\ell_2} \right\} \leq C \cdot \frac{\sigma (\sqrt{m}+\sqrt{n})}{\lambda \sqrt{k_m k_n}}.
$$
And according to the definition of the canonical angles, we have
$$
\max \left\{ \| U U^T - \tilde{U} \tilde{U}^T \|_{2}, \| V V^T - \tilde{V} \tilde{V}^T \|_{2} \right\} \leq C  \cdot \frac{\sigma (\sqrt{m}+\sqrt{n})}{\lambda \sqrt{k_m k_n}}.
$$

Now let us assume we have two observations of $X$. We use the first observation $\tilde{X}$ to solve for the singular vectors $\tilde{U},\tilde{V}$, we use the second observation $X$ to project to the singular vectors $\tilde{U},\tilde{V}$. We can use Tsybakov's sample cloning argument (\cite{tsybakovaggregation}, Lemma 2.1) to create two independent observations of X when noise is Gaussian as follows. Create a pure Gaussian matrix $Z'$ and define $X_1 = X + Z' = M + (Z + Z')$ and $X_2 = X - Z' = M + (Z - Z')$, making $X_1, X_2$ independent with the variance being doubled. This step is not essential because we can perform random subsampling as in \cite{vu2014simple}; having two observations instead of one does not change the picture statistically or computationally. Recall $X = M + Z = \lambda \sqrt{k_m k_n} U V^T +Z$.

Define the projection operator to be $\mathcal{P}$, we start the analysis by decomposing
\begin{align}
	\label{sp.eq}
\| \mathcal{P}_{\tilde{U}} X_{\cdot j}  - M_{\cdot j} \|_{\ell_2} \leq \| \mathcal{P}_{\tilde{U}} (X_{\cdot j}  - M_{\cdot j}) \|_{\ell_2} + \| (\mathcal{P}_{\tilde{U}} - I) M_{\cdot j} \|_{\ell_2}
\end{align}
for $1 \leq j \leq n$. 

For the first term of \eqref{sp.eq}, note that $X_{\cdot j} - M_{\cdot j} = Z_{\cdot j} \in \mathbb{R}^m$ is an i.i.d. isotropic sub-Gaussian vector, and thus we have through Lemma \ref{Proj.Lma}, for $t = (1+1/c) \log n$, $Z_{\cdot j} \in \mathbb{R}^m, 1\leq j\leq n$ and $r = 1$
\begin{align}
	\mathbb{P} \left(\| \mathcal{P}_{\tilde{U}} (X_{\cdot j}  - M_{\cdot j}) \|_{\ell_2} \geq \sigma \sqrt{r} \sqrt{1+ 2 \sqrt{1+1/c}\cdot \sqrt{\frac{\log n}{r}} + 2 (1+1/c) \cdot \frac{\log n}{r}} \right) \leq n^{-c-1}.
\end{align}
We invoke the union bound for all $1\leq j \leq n$ to obtain
\begin{align}
\max_{1\leq j \leq n} \| \mathcal{P}_{\tilde{U}} (X_{\cdot j}  - M_{\cdot j}) \|_{\ell_2} 
& \leq  \sigma \sqrt{r} + \sqrt{2(1+1/c)} \cdot \sigma \sqrt{\log n}  \\
& \leq \sigma + C \cdot \sigma \sqrt{\log n}
\end{align}
with probability at least $1 - n^{-c}$.

For the second term $M_{ \cdot j} = \tilde{X}_{ \cdot j} - \tilde{Z}_{ \cdot j}$ of \eqref{sp.eq}, there are two ways of upper bounding it. The first approach is to split
\begin{align}
	\label{tri.eq}
\| (\mathcal{P}_{\tilde{U}} - I) M \|_2 \leq \| (\mathcal{P}_{\tilde{U}} - I) \tilde{X} \|_2 + \| (\mathcal{P}_{\tilde{U}} - I) \tilde{Z} \|_2 \leq 2 \| \tilde{Z} \|_2.
\end{align}
The first term of \eqref{tri.eq} is $\sigma_2(\tilde{X}) \leq \sigma_2(M)+\| \tilde{Z} \|_2$ through Weyl's interlacing inequality, while the second term is bounded by $\| \tilde{Z} \|_2$. We also know that  $\|\tilde{Z}\|_2 \leq C_3 \cdot \sigma (\sqrt{m}+\sqrt{n})$. Recall the definition of the induced $\ell_2$ norm of a matrix $ (\mathcal{P}_{\tilde{U}} - I) M $:
\begin{align*}
\| (\mathcal{P}_{\tilde{U}} - I) M \|_2 \geq \frac{\|(\mathcal{P}_{\tilde{U}} - I) M V  \|_{\ell_2}}{\| V\|_{\ell_2}} = \|(\mathcal{P}_{\tilde{U}} - I) \lambda \sqrt{k_m k_n} U \|_{\ell_2} \geq \sqrt{k_n} \| (\mathcal{P}_{\tilde{U}} - I) M_{\cdot j} \|_{\ell_2}.
\end{align*}
In the second approach, the second term of \eqref{sp.eq} can be handled through perturbation Sin Theta Theorem \ref{lma:dkw}:
$$
\| (\mathcal{P}_{\tilde{U}} - I) M_{\cdot j} \|_{\ell_2} = \|(\mathcal{P}_{\tilde{U}} - \mathcal{P}_U) M_{\cdot j}  \|_{\ell_2} \leq \|\tilde{U} \tilde{U}^T - U U^T \|_2 \cdot \| M_{\cdot j} \|_{\ell_2} \leq  C \frac{\sigma \sqrt{m+n}}{\lambda \sqrt{k_m k_n}} \lambda \sqrt{k_m}.
$$
This second approach will be used in the multiple submatrices analysis.

Combining all the above, we have with probability at least $1 - n^{-c} - m^{-c}$, for all $1 \leq j \leq n$
\begin{align}
\| \mathcal{P}_{\tilde{U}} X_{\cdot j}  - M_{\cdot j} \|_{\ell_2} \leq C \cdot \left(\sigma \sqrt{\log n}+ \sigma \sqrt{\frac{m \vee n}{ k_n}} \right).
\end{align}
Similarly we have for all $1\leq i \leq m$, 
\begin{align}
\| \mathcal{P}_{\tilde{V}} X^T_{i \cdot }  - M^T_{i \cdot } \|_{\ell_2} \leq C \cdot \left( \sigma \sqrt{\log m}+ \sigma \sqrt{\frac{m \vee n}{ k_m}} \right).
\end{align}
Clearly we know that for $i \in R_m$ and $i' \in [m] \backslash R_m$
$$
\|  M^T_{i \cdot}-M^T_{i' \cdot} \|_{\ell_2}  = \lambda \sqrt{k_n}
$$
and for $j \in C_n$ and $j' \in [n] \backslash C_n$
$$
\|  M_{ \cdot j}-M_{ \cdot j'} \|_{\ell_2}  = \lambda \sqrt{k_m}
$$
Thus if 
\begin{align}
\lambda \sqrt{k_m} \geq 6C \cdot \left( \sigma \sqrt{\log n}+ \sigma \sqrt{\frac{m \vee n}{ k_n}} \right) \\
\lambda \sqrt{k_n} \geq 6C \cdot \left( \sigma \sqrt{\log m}+ \sigma \sqrt{\frac{m \vee n}{ k_m}} \right)
\end{align}
hold, then we have learned a metric $d$ (a one dimensional line) such that on this line, data forms  clusters in the sense that 
$$
2 \max_{i,i'\in R_m} |d_i - d_{i'}| \leq \min_{i \in R_m, i' \in [m] \backslash R_m} |d_i - d_{i'}|.
$$
In this case, a simple cut-off clustering recovers the nodes exactly.

In summary, if
$$
\lambda \geq C \cdot \sigma \left( \sqrt{\frac{\log n}{k_m }}+ \sqrt{\frac{\log m}{k_n}} +  \sqrt{\frac{m+n}{k_m k_n}} \right),
$$
the spectral algorithm succeeds with probability at least $$1 - m^{-c} - n^{-c} - 2 \exp\left(-c(m+n)\right).$$

\end{proof}

\begin{proof}[Proof of Lemma \ref{lma:thres-sp-alg}]
The proof of the validity of thresholded spectral algorithm at level $\sigma t$ is easy based on the proof of Lemma \ref{lma:sp-alg}. Firstly we have the following decomposition
\begin{align*}
	\eta_{\sigma t} (X) &= M + \eta_{\sigma t} (Z)  +  B\\
	&= \lambda 1_{R_m} 1_{C_n}^T +  \eta_{\sigma t} (Z)  +  B
\end{align*}
where $B$ is the bias matrix satisfying
\begin{align*}
	B_{ij} = 0, \quad \text{if}~(i,j) \notin R_m \times C_n \\
	|B_{ij}| \leq 2\sigma t, \quad \text{if}~(i,j) \in R_m \times C_n.
\end{align*}
Let us prove this fact. Clearly if $(i,j) \notin R_m \times C_n$, $B_{ij} = 0$. If $(i,j) \in R_m \times C_n$, we have
\begin{align*}
|B_{ij}| & = |\eta_{\sigma t}(\lambda + Z_{ij}) - \lambda  - \eta_{\sigma t}(Z_{ij})| \\
		 & \leq 	|\eta_{\sigma t}(\lambda + Z_{ij}) - (\lambda + Z_{ij})| + |(\lambda + Z_{ij}) -\lambda -  Z_{ij}| + |Z_{ij} -  \eta_{\sigma t}(Z_{ij})| \\
		 & \leq 2 \sigma t
\end{align*}
where the last step uses $|\eta_t(y) - y|\leq t$, for any $y$. 
Let us bound the variance of each thresholded entry $\eta_{\sigma t}(Z_{ij})$, 
\begin{align*}
	\mathbb{E} \left[ \eta_{\sigma t}(Z_{ij}) \right]^2 & = \int_{0}^\infty 2 z \cdot 2 \mathbb{P}( \eta_{\sigma t}(Z_{ij}) > z ) dz \\
	& = \int_{0}^\infty 4 z  \mathbb{P}( Z_{ij} > z + \sigma t  ) dz  \\
	& =  \int_{0}^{\infty}  4z \exp\left\{- c \cdot \frac{(z + \sigma t)^2}{2 \sigma^2} \right\} dz \\
	& \leq C \cdot \sigma^2 \cdot \exp(-\frac{t^2}{2})
\end{align*}
for some universal constant $C$. 
Clearly after thresholding, $\eta_{\sigma t}(Z)$ still have i.i.d entries, but the variance has been significantly reduced as $t \rightarrow \infty$. 

Via the perturbation analysis established in Proof of Lemma~\ref{lma:sp-alg}
\begin{align*}
 \| \eta_{\sigma t} (Z)  +  B \|_2 &\leq \| \eta_{\sigma t} (Z) \|_2 + \| B \|_2\\
 & \leq C \cdot \sigma \sqrt{m\vee n} \cdot \exp(-\frac{t^2}{2}) + 2 \sqrt{k_m k_n} \sigma t
\end{align*}
as $B$ only have $k_m k_n$ non zero entries. 
Thus applying Lemma \ref{lma:dkw}, we have
\begin{align*}
\max \left\{ |\sin \angle (U, \tilde{U}) |, |\sin \angle (V, \tilde{V})| \right\} & \leq \frac{C \cdot \sigma \sqrt{m\vee n} \cdot \exp(-\frac{t^2}{2}) + 2 \sqrt{k_m k_n} \sigma t}{\lambda \sqrt{k_m k_n} -  C \cdot \sigma \sqrt{m\vee n} \cdot \exp(-\frac{t^2}{2}) - 2 \sqrt{k_m k_n} \sigma t} \\
& \precsim \frac{\sqrt{m\vee n} \cdot \sigma \exp(-\frac{t^2}{2}) +  \sqrt{k_m k_n}  \sigma t}{ \lambda \sqrt{k_m k_n}}.
\end{align*}

As usual, we continue the analysis by decomposing (following the steps as in Lemma \ref{lma:dkw}, but with an additional bias term $B$)
\begin{align*}
\| \mathcal{P}_{\tilde{U}} \eta(X_{\cdot j})  - M_{\cdot j} \|_{\ell_2} & \leq \| \mathcal{P}_{\tilde{U}} (\eta(X_{\cdot j})  - M_{\cdot j}) \|_{\ell_2} + \| (\mathcal{P}_{\tilde{U}} - I) M_{\cdot j} \|_{\ell_2} \\
& \leq \| \mathcal{P}_{\tilde{U}} \eta(Z_{\cdot j}) \|_{\ell_2} + \| B_{\cdot j} \|_{\ell_2} + \| (\mathcal{P}_{\tilde{U}} - I) M_{\cdot j} \|_{\ell_2} \\
& \leq C\cdot \left\{ \sigma \exp(-\frac{t^2}{2}) \sqrt{\log n} + \sqrt{k_m} \sigma t + \frac{\sqrt{m\vee n} \cdot \sigma \exp(-\frac{t^2}{2}) +  \sqrt{k_m k_n}  \sigma t}{ \lambda \sqrt{k_m k_n}} \lambda \sqrt{k_m} \right\}
\end{align*}
for $1 \leq j \leq n$. 
We know for $j \in C_n$ and $j' \in [n] \backslash C_n$
$$
\|  M_{ \cdot j}-M_{ \cdot j'} \|_{\ell_2}  = \lambda \sqrt{k_m}
$$
Thus if 
\begin{align}
\lambda \sqrt{k_m} \geq 6C \cdot \left\{ \sigma \exp(-\frac{t^2}{2}) \sqrt{\log n} + \sqrt{k_m} \sigma t + \frac{\sqrt{m\vee n} \cdot \sigma \exp(-\frac{t^2}{2}) +  \sqrt{k_m k_n}  \sigma t}{  \sqrt{ k_n}}  \right\} 
\end{align}
hold, then we have learned a metric $d$ (a one dimensional line) such that on this line, data forms clusters.
In this case, a simple cut-off clustering recovers the nodes exactly.

In summary, if
$$
\frac{\lambda}{\sigma} \geq C \cdot  \left(  \left[ \sqrt{\frac{m \vee n}{k_m k_n}} + \sqrt{\frac{\log n}{k_m} \vee \frac{\log m}{k_n}} \right] \cdot e^{-t^2/2} + t \right),
$$
the thresholded spectral algorithm succeeds with probability at least $$1 - m^{-c} - n^{-c} - 2 \exp\left(-c(m+n)\right).$$

\end{proof}

\begin{proof}[Proof of Theorem \ref{CLL.Thm}]

	 Computational lower bound for localization (support recovery) is of different nature than the computational lower bound for detection (two point testing). The idea is to design a randomized polynomial time algorithmic reduction to relate a an instance of \emph{hidden clique} problem to our submatrix localization problem. 
	  The proof proceeds in the following way: we will construct a randomized polynomial time transformation $\mathcal{T}$ to map a random instance of $\mathcal{G}(N,\kappa)$ to a random instance of our submatrix $\mathcal{M}(m=n,k_m\asymp k_n\asymp k,\lambda/\sigma)$ (abbreviated as $\mathcal{M}(n,k,\lambda/\sigma)$). Then we will provide a quantitative computational lower bound by showing that if there is a polynomial time algorithm that pushes below the hypothesized computational boundary for localization in the submatrix model, there will be a polynomial time algorithm that solves hidden clique localization with high probability (a contradiction to $\sf HC_{l}$).
	 
	 

	 Denote the randomized polynomial time transformation as $$\mathcal{T}: \mathcal{G}(N, \kappa(N)) \rightarrow M(n,k = n^{\alpha},\lambda/\sigma = n^{-\beta}).$$ There are several stages for the construction of the algorithmic reduction. First we define a graph $\mathcal{G}^{e}(N, \kappa(N))$ that is stochastically equivalent to the hidden clique graph $\mathcal{G}(N,\kappa(N))$, but is easier for theoretical analysis.  $\mathcal{G}^e$ has the property: each node independently has the probability $\kappa(N)/N$ to be a clique node, and with the remaining probability a non-clique node. Using Bernstein's inequality and the inequality  \eqref{eq:bern-trick} proved below. with probability at least $1 - 2N^{-1}$ the number of clique nodes  $\kappa^{e}$ in $\mathcal{G}^e$
	 \begin{align}
	 \kappa \left(1 - \sqrt{\frac{4 \log N}{\kappa}} \right) \leq \kappa^e \leq  \kappa \left(1 + \sqrt{\frac{4 \log N}{\kappa}} \right) \Rightarrow \kappa^e \asymp \kappa
	 \end{align}
	 as long as $\kappa \succsim \log N$.

	 Consider a hidden clique graph $\mathcal{G}^e(2N,2\kappa(N))$ with $N = n$ and $\kappa(N) = \kappa$. Denote the set of clique nodes for $\mathcal{G}^e(2N,2\kappa(N))$ to be $C_{N,\kappa}$. Represent the hidden clique graph using the symmetric adjacency matrix $G \in \{-1,1\}^{2N \times 2N}$, where $G_{ij} = 1$ if $i,j \in C_{N,\kappa}$, otherwise with equal probability to be either $-1$ or $1$. As remarked before, with probability at least $1 - 2N^{-1}$, we have planted $2 \kappa(1\pm o(1))$ clique nodes in graph $\mathcal{G}^e$ with $2N$ nodes. Take out the upper-right submatrix of $G$, denote as $G_{UR}$ where $U$ is the index set $1\leq i \leq N$ and $R$ is the index set $N+1 \leq j \leq 2N$. Now $G_{UR}$ has independent entries. 
	 
	 The construction of $\mathcal{T}$ employs the \emph{Bootstrapping} idea. Generate $l^2$ (with $l \asymp n^{\beta}, 0<\beta<1/2$) matrices through bootstrap subsampling as follows. Generate $l-1$ independent index vectors $\psi^{(s)} \in \mathbb{R}^n, 1\leq s<l$, where each element $\psi^{(s)}(i), 1\leq i\leq n$ is a random draw with replacement from the row indices $[n]$. Denote vector $\phi^{(0)}(i)=i, 1\leq i\leq n$ as the original index set. Similarly, we can define independently the column index vectors $\phi^{(t)}, 1\leq t<l$. We remark that these bootstrap samples can be generated in polynomial time $\Omega(l^2 n^2)$. The transformation is a weighted average of $l^2$ matrices of size $n\times n$ generated based on the original adjacency matrix $G_{UR}$.

	 \begin{align}
		\mathcal{T}:~~ M_{ij} = \frac{1}{l} \sum_{0\leq s,t <l} (G_{UR})_{\psi^{(s)}(i) \phi^{(t)}(j)}, ~~ 1\leq i,j \leq n.
	 \end{align}
	 Due to the bootstrapping property, the matrices $\left[(G_{UR})_{\psi^{(s)}(i) \phi^{(t)}(j)}\right]_{1\leq i,j\leq n}$, indexed by $0\leq s,t<l$ are independent of each other. 
	 Recall that $C_{N,\kappa}$ stands for the clique set of the hidden clique graph. We define the row candidate set $R_{l} : = \{ i \in [n]: \exists~ 0\leq s<l, \psi^{(s)}(i) \in C_{N,\kappa}  \}$ and column candidate set $C_{l} := \{ j \in [n]: \exists~ 0\leq t<l, \phi^{(t)}(j) \in C_{N,\kappa}\}$. Observe that  $R_{l} \times C_{l}$ are the indices where the matrix $M$ contains signal.

	 There are two cases for $M_{ij}$, given the candidate set $R_{l}\times C_{l}$. If $i \in R_{l}$ and $j \in C_{l}$, namely when $(i,j)$ is a clique edge in at least one of the $l^2$ matrices, then $\mathbb{E}[ M_{ij} | \mathcal{G}^e] \geq l^{-1}$ where the expectation is taken over the bootstrap $\sigma$-field conditioned on the candidate set $R_{l}\times C_{l}$ and the original $\sigma$-field of $\mathcal{G}^e$. Otherwise $\mathbb{E} [M_{ij}| \mathcal{G}^e]  = l(\frac{|E|}{N^2 - \kappa^2}-\frac{1}{2}) $ for $(i,j) \notin R_{l}\times C_{l}$, where $|E|$ is a $\text{Binomial}(N^2 - \kappa^2, 1/2)$. With high probability, $\mathbb{E} [M_{ij}| \mathcal{G}^e]  \asymp \frac{l}{\sqrt{N^2 - \kappa^2}} \asymp \frac{l}{n} = o(\frac{1}{l})$. Thus the mean separation between the signal position and non-signal position is $\frac{1}{\ell} -  \frac{l}{n} \asymp \frac{1}{\ell}$. Note in the submatrix model, it does not matter if the noise has mean zero or not (since we can subtract the mean)-- only the signal separation matters. 
	 
Now let us discuss the independence issue in $M$ through our Bootstrapping construction. Clearly due to sampling with replacement and bootstrapping, condition on $\mathcal{G}^e$, we have independence among samples for the same location $(i,j)$
	$$
	(G_{UR})_{\psi^{(s)}(i) \phi^{(t)}(j)} \perp (G_{UR})_{\psi^{(s')}(i) \phi^{(t')}(j)}.
	$$
	For the independence among entries in one Bootstrapped matrix, clearly
	$$
	(G_{UR})_{\psi^{(s)}(i) \phi^{(t)}(j)} \perp (G_{UR})_{\psi^{(s)}(i') \phi^{(t)}(j')}.
	$$
	The only case where there might be a slight dependence is between $(G_{UR})_{\psi^{(s)}(i) \phi^{(t)}(j)}$ and $(G_{UR})_{\psi^{(s)}(i) \phi^{(t)}(j')}$. The way to eliminate the slight dependence is through \cite{vu2008random}'s result on universality of random discrete graphs. \cite{vu2008random} showed random regular graph $\mathcal{G}(n,n/2)$ shares many similarities as Erd\H{o}s-R\'{e}nyi random graph $\mathcal{G}(n, 1/2)$, for instance, top and second eigenvalues ($n/2$ and $\sqrt{n}$ respectively), limiting spectral distribution, sandwich conjecture, determinant, etc. Let us consider the case where the upper-right of the adjacency matrix $G$ consists of random bi-regular graph (see \cite{deshpande2013finding} for difficulty of clique problem under random regular graph) with degree $n/2$ instead of the Erd\H{o}s-R\'{e}nyi graph. The only thing we need to change is assuming hidden clique hypothesis is still valid for the following random graph: for a $n\times n$ adjacency matrix $G$, first find a clique/principal submatrix of size $k$ uniformly randomly and connect density, for the remaining part of the matrix, sample a random regular graph of $G(n-k, \frac{n-k}{2})$ and a random bi-regular graph of size $k \times (n-k)$ with left regular degree $n/2-k$ and right regular degree $k/2$ (here degree test will not work in this graph and spectral barrier still suggests $k \precsim \sqrt{n}$ is hard due to universality result of random discrete graphs). In the bootstrapping step, condition on the same row $\psi^{(s)}(i)$ being not a clique, 
$(G_{UR})_{\psi^{(s)}(i) \phi^{(t)}(j)} \perp (G_{UR})_{\psi^{(s)}(i) \phi^{(t)}(j')} | \psi^{(s)}(i)$, and each one is a Rademacher random variable (regardless of the choice of $\psi^{(s)}(i)$), which implies $(G_{UR})_{\psi^{(s)}(i) \phi^{(t)}(j)} \perp (G_{UR})_{\psi^{(s)}(i) \phi^{(t)}(j')}$ holds unconditionally. Thus in the bootstrapping procedure, we have independence among entries within the matrix.

	 Let us move to verify the sub-Gaussianity of $M$ matrix. Note that for the index $i,j$ that is not a clique for any of the matrices, $M_{ij}$ is sub-Gaussian, due to Hoeffding's inequality
	 \begin{align}
		 \mathbb{P}\left(|M_{ij} - \mathbb{E}M_{ij}| \geq u \right) \leq 2 \exp (- u^2/2).
	 \end{align}
	 For the index $i,j$ being a clique in at least one of the matrices, we claim the number of matrices has $(i,j)$ being clique is $O^*(1)$. Due to Bernstein's inequality, we have $\max_{i} |\{ 0\leq s <l:\psi^{(s)}(i) \in C_{N,\kappa} \}| \leq \frac{\kappa l}{n} + \frac{8}{3} \log n$ with probability at least $1- n^{-1}$. This further implies there are at least $l^2 - (\frac{\kappa l}{n} + \frac{8}{3} \log n)^2$ many independent Rademacher random variables in each $i,j$ position, thus
	 \begin{align}
		 \mathbb{P}\left(|M_{ij} - \mathbb{E}M_{ij}| \geq u \right) \leq 2 \exp \left(- (1 - C \cdot (\kappa n^{-1} + l^{-1} \log n)^2)u^2/2 \right).
	 \end{align}
	 Up to now we have proved that when $i,j$ is a signal node for $M$, then $O^*(1) l^{-1} \geq \mathbb{E} M_{ij} \geq l^{-1}$.
	 Thus we can take sub-Gaussian parameter to be any $\sigma < 1$ because $\kappa n^{-1}, l^{-1} \log n$ are both $o(1)$. The constructed $M(n,k,\lambda/\sigma)$ matrix satisfies the submatrix model with $\lambda/\sigma = l^{-1}$ and sub-Gaussian parameter $\sigma = 1 - o(1)$.  
	 
	 Let us estimate the corresponding $k$ in the submatrix model.  We need to bound the order of the cardinality of $R_{l}$, denoted as $|R_{l}|$. The total number of positions with signal (at least one clique node inside) is $$\mathbb{E} |R_l| = \mathbb{E} |\{ 1\leq i\leq n: i \in R_l \}| = n \left[1-(1 - \kappa/n)^{l} \right].$$ Thus we have the two sided bound
	 $$
	 \kappa l \left(1 - \frac{\kappa l}{2n} \right)  \leq \mathbb{E}|R_l| \leq  \kappa l
	 $$
	 which is of the order $k:=\kappa l$. Let us provide a high probability bound on $|R_l|$. By Bernstein's inequality
	 \begin{align}
	 	\mathbb{P}\left(\left| |R_l| - \mathbb{E} |R_l| \right| > u \right) \leq 2 \exp\left( - \frac{u^2/2}{ \kappa l +u/3} \right).
	 \end{align}
	 Thus if we take $u = \sqrt{4 \kappa l \log n }$, as long as $\log n = o(\kappa l)$,
	 \begin{align}
		\label{eq:bern-trick}
	 	\mathbb{P}\left( \left| |R_l| - \mathbb{E}|R_l| \right| >  \sqrt{4 \kappa l \log n} \right) \leq 2 n^{-1}.
	 \end{align}
	 So with probability at least $1- 2n^{-1}$, the number of positions that contain signal nodes is bounded as
	 \begin{align}
		 \label{bn.eq}
	   \kappa l \left(1 - \frac{\kappa l}{n} \right) \left(1 - \sqrt{\frac{4\log n}{\kappa l}} \right) < |R_l| < \kappa l \left(1 + \sqrt{\frac{4\log n}{\kappa l}} \right) \Rightarrow |R_l| \asymp \kappa l.
	 \end{align} 
	 Equation \eqref{bn.eq} implies that with high probability
	 \begin{align*}
	 \kappa l (1 - o(1))  \leq | R_l | \leq  \kappa l (1+o(1)), \\
	  \kappa l (1 - o(1))  \leq | C_l | \leq  \kappa l (1+o(1)).
	 \end{align*}
	 The above means, in the submatrix parametrization, $k_m \asymp k_n \asymp \kappa l \asymp n^{\alpha}$, $\lambda/\sigma \asymp l^{-1} \asymp n^{-\beta}$, which implies $\kappa \asymp n^{\alpha-\beta}$.
	 
	 Suppose there exists a polynomial time algorithm $\mathcal{A}_M$ that pushes below the computational boundary. In other words,
	 \begin{align}
	n^{-\beta}\asymp \frac{\lambda}{\sigma} \precsim \sqrt{\frac{m+n}{k_m k_n}}  \asymp n^{(1-2\alpha)/2} \Rightarrow \beta > \alpha-\frac{1}{2} 
	 \end{align}
	 with the last inequality having a slack $\epsilon>0$. More precisely, $\mathcal{A}_M$ returns two estimated index sets $\hat{R}_n$ and $\hat{C}_n$ corresponding to the location of the submatrix (and correct with probability going to $1$) under the regime $\beta = \alpha-1/2+\epsilon$. Suppose under some conditions, this algorithm $\mathcal{A}_M$ can be modified to a randomized polynomial time algorithm $\mathcal{A}_{\mathcal{G}}$ that correctly identifies the hidden clique nodes with high probability.
	 It means in the corresponding hidden clique graph $\mathcal{G}(2N, 2\kappa )$, $\mathcal{A}_{\mathcal{G}}$ also pushes below the computational boundary of hidden clique by the amount $\epsilon$:
	 \begin{align}
	 \kappa(N) =  2\kappa \asymp (2n)^{\alpha-\beta} \asymp n^{1/2-\epsilon} \precsim n^{1/2} \asymp N^{\frac{1}{2}}.
	 \end{align}
	 In summary, the quantitative computational lower bound implies that if the computational boundary for submatrix localization is pushed below by an amount $\epsilon$ in the power, the \emph{hidden clique} boundary is correspondingly improved by $\epsilon$. 
	 
	 Now let us show that any algorithm $\mathcal{A}_M$ that localizes the submatrix introduces a randomized algorithm that finds the hidden clique nodes with probability tending to 1. The algorithm relies on the following simple lemma.
	 \begin{lemma}
	 \label{lma:neig}
	 For the hidden clique model $\mathcal{G}(N,\kappa)$, suppose an algorithm provides a candidate set $S$ of size $k$ that contains the true clique subset exactly. If
	 \begin{align*}
	 \kappa \geq C \sqrt{k \log N}
	 \end{align*}
	 then by looking at the adjacency matrix restricted to $S$ we can recover the clique subset exactly with high probability.
	 \end{lemma}
	 The proof of Lemma \ref{lma:neig} is immediate. If $i$ is a clique node, then $\min_{i} \sum_{j \in C} G_{ij} \geq \kappa - C/2\cdot \sqrt{k \log N}$. If $i$ is not a clique node, then $\max_{i} \sum_{j \in C} G_{ij} \leq C/2\cdot \sqrt{k \log N}$. The proof is completed.

	  Algorithm $\mathcal{A}_M$ provides candidate sets $R_l, C_l$ of size $k$, inside which $\kappa$ are correct clique nodes, and thus exact recovery can be completed through Lemma~\ref{lma:neig} since $\kappa \succsim (k\log N)^{1/2}$ (since $\kappa \asymp n^{1/2-\epsilon} \succsim k^{1/2} \asymp n^{\alpha/2}$ when $\epsilon$ is small). The algorithm $\mathcal{A}_{M}$ induces another randomized polynomial time algorithm $\mathcal{A}_{\mathcal{G}}$ that solves the hidden clique problem $\mathcal{G}(2N, 2\kappa)$ with $\kappa \precsim N^{1/2}$. The algorithm $\mathcal{A}_{\mathcal{G}}$ returns the support $\hat{C}_{N,\kappa}$ that coincides with the true support $C_{N,\kappa}$ with probability going to $1$ (a contradiction to the hidden clique hypothesis $\sf HC_{l}$). We conclude that, under the hypothesis, there is no polynomial time algorithm $\mathcal{A}_M$ that can push below the computational boundary $\lambda \precsim \sqrt{\frac{m+n}{k_m k_n}}$.

\end{proof}

\begin{proof}[Proof of Lemma \ref{lma:info-low}]
	The proof of this lemma uses the well-known Fano's information inequality, namely Lemma \ref{lma:fano}. We have that $X = M + Z$, where $M \in \mathbb{R}^{m \times n}$ is the mean matrix. Under the Gaussian noise with parameter $\sigma$, the probability model is
\begin{align}
\mathbb{P}(X|M) \propto \exp\left( - \langle X - M, X - M \rangle /2 \sigma^2 \right)
\end{align}
where $M = \lambda \cdot U V^T \in \Theta$. The parameter space $\Theta$ is composed of all $M = \lambda \cdot U V^T$ where $U$ are sampled uniformly on the collection of vectors with $k_m$ ones and other coordinates being zero, and similarly $V$ are sampled uniformly with $k_n$ ones and the rest zero. The cardinality of the parameter space is 
$$
{\sf Card}(\Theta) = \dbinom{m}{k_m} \dbinom{n}{k_n},
$$
corresponding to that many probability measures on the same probability space. Put a uniform prior on this parameter space and invoke Fano's lemma \ref{lma:fano}. To obtain the lower bound, we need to upper bound the Kullback-Leibler divergence $\KL(\mc{P}_{M} || \bar{\mc{P}})$ for any $M \in \Theta$, where
$$
\bar{\mc{P}}  = \mathbb{E}_{M' \sim {\sf unif}(\Theta)} \mc{P}_{M'}.
$$
For any $M \in \Theta$,
\begin{align*}
    \KL(\mc{P}_{M} || \bar{\mc{P}}) & = \mathbb{E}_{\mc{P}_{M}} \log \frac{\mc{P}_{M}}{\bar{\mc{P}}} \\
                                      & = \mathbb{E}_{\mc{P}_{M}} \log \frac{\mc{P}_{M}}{ \mathbb{E}_{M' \sim {\sf unif}(\Theta)} \mc{P}_{M'}} \\
                                      & \leq \mathbb{E}_{X \sim \mc{P}_{M}} \left\{ -\frac{\langle X - M, X - M \rangle}{2\sigma^2} + \frac{1}{{\sf Card}(\Theta)} \sum _{M' \in \Theta} \frac{\langle X - M', X - M' \rangle}{2\sigma^2} \right\} \\
                                      & = \mathbb{E}_{X \sim \mc{P}_{M}} \left\{ \frac{1}{{\sf Card}(\Theta)} \sum _{M' \in \Theta} 2 \frac{\langle M - M', X - M \rangle}{2\sigma^2} + \frac{1}{{\sf Card}(\Theta)} \sum _{M' \in \Theta} \frac{\langle M - M', M - M' \rangle}{2\sigma^2} \right\} \\
                                      & = \frac{1}{{\sf Card}(\Theta)} \sum _{M' \in \Theta} \frac{\langle M - M', M - M' \rangle}{2\sigma^2} \\
                                      & = \frac{ \langle M,M \rangle + \frac{1}{{\sf Card}(\Theta)} \sum _{M' \in \Theta}  \langle M',M' \rangle - 2\frac{1}{{\sf Card}(\Theta)} \sum _{M' \in \Theta} \langle M,M' \rangle}{2\sigma^2} \\
                                      & = \frac{\lambda^2 k_m k_n}{\sigma^2} (1 - \frac{k_m k_n}{mn}).
\end{align*}
Thus as long as
\begin{align}
\frac{\lambda^2 k_m k_n}{\sigma^2} (1 - \frac{k_m k_n}{mn}) \leq \alpha \log \left[ \dbinom{m}{k_m} \dbinom{n}{k_n} \right]
\end{align}
we have
\begin{align}
	\frac{1}{{\sf Card}(\Theta)} \sum_{M \in \Theta} \KL(\mc{P}_{M} || \bar{\mc{P}}) \leq \alpha \log ({\sf Card}(\Theta)).
\end{align}
Invoke the simple bound on binomial coefficients $\left(\frac{n}{k}\right)^k \leq \binom{n}{k} \leq \left(\frac{ne}{k}\right)^k$. If we choose
\begin{align}
\lambda \leq  C_{\alpha} \cdot  \sigma \sqrt{\frac{k_m \log \frac{m}{k_m} + k_n \log \frac{n}{k_n}}{k_m k_n}}
\end{align}
then the condition \eqref{Fano.Cnd} holds. Any submatrix localization algorithm translates into a multiple testing procedure that picks a parameter $M' \in \Theta$. By Fano's information inequality \ref{lma:fano}, the minimax error, which is also the localization error, is at least $1 - \alpha - \frac{\log 2}{k_m \log(m/k_m) + k_n \log (n/k_n)}$.
\end{proof}

\begin{proof}[Proof of Lemma \ref{lma:stat-alg}]
	Recall the definition \ref{SubG.Def} of a sub-Gaussian random variable. Taking $Z_{ij}, i\in I, j \in J$, we have the following concentration from the Chernoff's bound for $\sum_{i\in I, j\in J} Z_{ij}$
	\begin{align}
		\mathbb{E} e^{\lambda \sum_{i\in I, j\in J} Z_{ij}} \leq \exp\left(|I||J| \cdot \sigma^2 \lambda^2/2c \right) 
	\end{align}
	and
	\begin{align}
		\mathbb{P}\left( |\sum_{i\in I, j\in J} Z_{ij}| \geq \sqrt{|I||J|} \sigma t \right) \leq 2 \exp( - c \cdot t^2/2).
	\end{align}
	There are in total
	$$
	\binom{m}{k_m} \binom{n}{k_n}
	$$
	such submatrices, so by a union bound, we have
	\begin{align*}
		\mbb{P}\left( \max_{|I|=R_m, |J| = C_n} |\sum_{i\in I, j\in J} Z_{ij}| \geq \sqrt{|I||J|} \sigma t \right) 
		& \leq \binom{m}{k_m} \binom{n}{k_n} \cdot 2  \exp( - c \cdot t^2/2) \\
		& \leq  2 \exp \left(k_m \log(e m/k_m) + k_n \log (e n/k_n) - c t^2/2\right).                      \end{align*}
	If we take $t = 2 \sqrt{\frac{k_m \log(e m/k_m) + k_n \log (e n/k_n) }{c}}$, then with probability at least 
	$$
	1 - 2 \left( \binom{m}{k_m} \binom{n}{k_n} \right)^{-1}
	$$
	we have
	\begin{align}
		\max_{|I|=R_m, |J| = C_n} |\sum_{i\in I, j\in J} Z_{ij}| \leq \frac{2}{\sqrt{c}} \sqrt{|I||J|} \sigma \sqrt{ k_m \log(e m/k_m) + k_n \log (e n/k_n) }.
	\end{align}
	Thus if 
	\begin{align}
	    k_m k_n \lambda > 2 \max_{|I|=k_m, |J| = k_n} |\sum_{i\in I, j\in J} Z_{ij}|
	\end{align}
	then the maximum submatrix is unique and is the true one. Recollecting terms, we reach
	\begin{align}
		\label{mono1.eq}
		\lambda > \frac{4}{\sqrt{c}} \cdot \sigma \sqrt{\frac{\log (em/k_m)}{k_n}+\frac{\log (en/k_n)}{k_m}}.
	\end{align}
	
	To make the proof fully rigorous, we need the following monotonicity trick. Consider the submatrix of size $k_m k_n$ with $a$ rows to be in the correct set $R_m$ and $b$ columns to be in the correct set $C_n$, where $a < k_m$ and $b < k_n$. The cardinality of the set of such matrices is
	$$
	\binom{m-a}{k_m-a}  \binom{n-a}{k_n-b}.
	$$ 
	Using the same calculation as before we want
	\begin{align}
		 \label{mono2.eq}
	     \lambda > \frac{4}{\sqrt{c}} \cdot \sigma \sqrt{\frac{(k_m-a) \log (e(m-a)/(k_m-a)) + (k_n-b) \log (e(n-b)/(k_n-b))}{k_m k_n-ab}}.
	\end{align}
	By simple algebra,
	\begin{align}
		\frac{k_m-a}{k_m k_n -ab} & < \frac{1}{k_n} \\
		\frac{k_n-b}{k_m k_n -ab} & < \frac{1}{k_m} \\
		 \log (e(m-a)/(k_m-a)) & < \log (em) \\
		 \log (e(n-b)/(k_n-b)) & < \log (en).
	\end{align}
	Hence, if equation \eqref{mono1.eq} is satisfied, \eqref{mono2.eq} is satisfied up to a universal constant for all $a < k_m$ and $b < k_n$. Thus we have proved that if
	$$
	\lambda \geq C \cdot \sigma \sqrt{\frac{ \log(em/k_m) }{k_n} + \frac{ \log (en/k_n)}{k_m}}
	$$
	with a suitable constant $C$, the statistical search algorithm picks out the correct submatrix. The sum of the probabilities of the bad events is bounded by
	$$
	\sum_{0\leq a<k_m,0\leq b<k_n} \left( \binom{m-a}{k_m-a} \binom{n-b}{k_n-b} \right)^{-1} \leq \frac{k_m k_n}{(m-k_m)(n-k_n)}.
	$$
	For the multiple non-overlapping submatrices case, as long as 
	$$
	m - \sum_{s=1}^r k_s^{(m)} \asymp m \quad n - \sum_{s = 1}^r k_s^{(n)} \asymp n
	$$
	then sequential application of Algorithm~\ref{SS.Alg} will find the $r$-submatrices. 
	
\end{proof}

\begin{proof}[Proof of Lemma \ref{lma:sp-ext} for Multiple Non-overlapping Submatrices Case]
We are going to provide theoretical justification to the extension of the submatrix localization algorithm to multiple non-overlapping submatrices case as in Algorithm \ref{SVD2.Alg}. Write out the matrix form of the submatrix model, with the SVD version of the signal matrix $M$
$$
X =  M + Z = U \Lambda V^T + Z.
$$
Due to the non-overlapping property, we have $1 \leq s \neq t \leq r$, $1_{R_s}^T 1_{R_t} = 0$, so as to $C_n$. The singular values of $U \Lambda V^T$ are $\lambda_s \sqrt{k_s^{(m)} k_s^{(n)}}, 1\leq s \leq r$, and all the other singular values are $0$.

Let us apply the Davis-Kahan-Wedin bound to $X =U \Lambda V^T +Z$. Denote the top $r$ left and right singular vector of $X$ as $\tilde{U}$ and $\tilde{V}$. Using Weyl's interlacing inequality, we have
$$
| \sigma_s(X)  - \sigma_s(M) | \leq \| Z \|_2
$$
and
$$
\sigma_s(X) \geq \lambda_s \sqrt{k_s^{(m)} k_s^{(n)}}  - \| Z \|_{2}, ~~ 1\leq s \leq r ;
$$
$$
\sigma_t(X) \leq \| Z \|_{2}, ~~ t > r.
$$
Thus applying Lemma \ref{lma:dkw}, we have
\begin{align*}
\max \left\{ |\sin \angle (U, \tilde{U}) |, |\sin \angle (V, \tilde{V})| \right\}  &\leq \frac{C \sigma (\sqrt{m}+\sqrt{n})}{\min\limits_{1\leq s \leq r} \lambda_s \sqrt{k_s^{(m)} k_s^{(n)}} -  C \sigma (\sqrt{m}+ \sqrt{n})} \\
&\asymp \frac{\sigma (\sqrt{m}+\sqrt{n})}{\min\limits_{1\leq s \leq r} \lambda_s \sqrt{k_s^{(m)} k_s^{(n)}}}.
\end{align*}
According to the definition of the canonical angles, we have
$$
\max \left\{ \| U U^T - \tilde{U} \tilde{U}^T \|_{2}, \| V V^T - \tilde{V} \tilde{V}^T \|_{2} \right\} \leq C  \cdot \frac{\sigma (\sqrt{m}+\sqrt{n})}{\min\limits_{1\leq s \leq r} \lambda_s \sqrt{k_s^{(m)} k_s^{(n)}}}.$$

Now let us assume we have two observation of $X$. We use the first observation $\tilde{X}$ to solve for the singular vectors $\tilde{U},\tilde{V}$, we use the second observation $X$ to project the to the singular vectors $\tilde{U},\tilde{V}$. 
Recall $X = M + Z =  U \Lambda V^T +Z$. For $1 \leq j \leq n$
\begin{align}
\| \mathcal{P}_{\tilde{U}} X_{\cdot j}  - M_{\cdot j} \|_{\ell_2} \leq \| \mathcal{P}_{\tilde{U}} (X_{\cdot j}  - M_{\cdot j}) \|_{\ell_2} + \| (\mathcal{P}_{\tilde{U}} - I) M_{\cdot j} \|_{\ell_2}
\end{align}

For the first term of \eqref{sp.eq} because $X_{\cdot j} - M_{\cdot j} = Z_{\cdot j} \in \mathbb{R}^m$ is an i.i.d. isotropic sub-Gaussian vector, we have through Lemma \ref{Proj.Lma}, for $t = (1+1/c) \log n$, $Z_{\cdot j} \in \mathbb{R}^m, 1\leq j\leq n$ and $r >0$
\begin{align}
	\mathbb{P} \left(\| \mathcal{P}_{\tilde{U}} (X_{\cdot j}  - M_{\cdot j}) \|_{\ell_2} \geq \sigma \sqrt{r} \sqrt{1+ 2 \sqrt{1+1/c}\cdot \sqrt{\frac{\log n}{r}} + 2 (1+1/c) \cdot \frac{\log n}{r}} \right) \leq n^{-c-1}.
\end{align}
Thus invoke the union bound for all $1\leq j \leq n$
\begin{align}
\max_{1\leq j \leq n} \| \mathcal{P}_{\tilde{U}} (X_{\cdot j}  - M_{\cdot j}) \|_{\ell_2} 
& \leq  \sigma \sqrt{r} + \sqrt{1+1/c} \cdot \sigma \sqrt{\log n} + (1+1/c) \frac{\log n}{\sqrt{r}} \\
& \leq \sigma \sqrt{r} + C \cdot \sigma \sqrt{\log n}
\end{align}
with probability at least $1 - n^{-c}$.

For the second term $M_{ \cdot j} = \tilde{X}_{ \cdot j} - \tilde{Z}_{ \cdot j}$ of \eqref{sp.eq}. It can be estimated through perturbation Sin Theta Theorem \ref{lma:dkw}. Basically it is 
\begin{align}
\| (\mathcal{P}_{\tilde{U}} - I) M_{\cdot j} \|_{\ell_2} & = \|(\mathcal{P}_{\tilde{U}} - \mathcal{P}_U) M_{\cdot j}  \|_{\ell_2} \leq \|\tilde{U} \tilde{U}^T - U U^T \|_2 \cdot \| M_{\cdot j} \|_{\ell_2} \\
& \leq  C \frac{\sigma \sqrt{m+n}}{\min\limits_{1\leq s \leq r} \lambda_s \sqrt{k_s^{(m)} k_s^{(n)}}}  \max\limits_{1 \leq s \leq r}\lambda_s \sqrt{k_s^{(m)}}.
\end{align}

Combining all the above, we have with probability at least $1 - n^{-c} - m^{-c}$, for all $1 \leq j \leq n$
\begin{align}
\| \mathcal{P}_{\tilde{U}} X_{\cdot j}  - M_{\cdot j} \|_{\ell_2} \leq C \cdot \left(\sigma \sqrt{r} + \sigma \sqrt{\log n}+ \sigma \sqrt{m \vee n} \cdot \frac{\max\limits_{1 \leq s \leq r}\lambda_s \sqrt{k_s^{(m)}}}{\min\limits_{1\leq s \leq r} \lambda_s \sqrt{k_s^{(m)} k_s^{(n)}}} \right).
\end{align}
Similarly we have for all $1\leq i \leq m$, 
\begin{align}
\| \mathcal{P}_{\tilde{V}} X^T_{i \cdot }  - M^T_{i \cdot } \|_{\ell_2} \leq C \cdot \left(\sigma \sqrt{r} + \sigma \sqrt{\log m}+ \sigma \sqrt{m \vee n} \cdot \frac{\max\limits_{1 \leq s \leq r}\lambda_s \sqrt{k_s^{(n)}}}{\min\limits_{1\leq s \leq r} \lambda_s \sqrt{k_s^{(m)} k_s^{(n)}}}  \right)
\end{align}
Clearly we know for any $1\leq s\leq r$ and $i \in R_s$ and $i' \in [m] \backslash R_s$
$$
\|  M^T_{i \cdot}-M^T_{i' \cdot} \|_{\ell_2}  \geq  \min\limits_{1\leq s\leq r} \lambda_s \sqrt{k_s^{(n)}}
$$
and for any $1\leq s\leq r$ and $j \in C_s$ and $j' \in [n] \backslash C_s$
$$
\|  M_{ \cdot j}-M_{ \cdot j'} \|_{\ell_2}  \geq  \min\limits_{1\leq s\leq r} \lambda_s \sqrt{k_s^{(m)}}.
$$
 
Thus if $k_s^{(m)} \asymp k_m, k_s^{(n)} \asymp k_n$ and $\lambda_s \asymp \lambda$ for all $1\leq s \leq r$
\begin{align}
\lambda \sqrt{k_m} \geq 6C \cdot \left( \sigma \sqrt{r} + \sigma \sqrt{\log n}+ \sigma \sqrt{\frac{m \vee n}{ k_n}} \right) \\
\lambda \sqrt{k_n} \geq 6C \cdot \left( \sigma \sqrt{r} + \sigma \sqrt{\log m}+ \sigma \sqrt{\frac{m \vee n}{ k_m}} \right)
\end{align}
We have learned a metric $d$ (of intrinsic dimension $r$) such that under this metric, data forms into clusters in the sense that 
$$
2 \max_{i,i'\in R_s} |d_i - d_{i'}| \leq \min_{i \in R_m, i' \in [m] \backslash R_s} |d_i - d_{i'}|.
$$
Thus it satisfies the geometric separation property.

Thus in summary if
$$
\lambda \geq C \cdot \sigma \left( \sqrt{\frac{r}{k_m \wedge k_n}} +  \sqrt{\frac{\log n}{k_m }}\vee \sqrt{\frac{\log m}{k_n}} +  \sqrt{\frac{m\vee n}{k_m k_n}} \right),
$$
the spectral algorithm succeeds with probability at least $$1 - m^{-c} - n^{-c} - 2 \exp\left(-c(m+n)\right).$$
Due to the fact that 
$$
\sqrt{\frac{r}{k_m \wedge k_n}} \precsim \sqrt{\frac{m\vee n}{k_m k_n}} 
$$
because $r k_m \precsim m, ~r k_n \precsim n$ in most cases, the first term does not have an effect in. most cases.
\end{proof}

Proof of Theorem~\ref{thm:comp-bdr} is a direct result of Lemma~\ref{lma:sp-alg} and Theorem~\ref{CLL.Thm}.
Proof of Theorem~\ref{thm:comp-bdr-sparse} is obvious based on Lemma~\ref{lma:thres-sp-alg} and the hidden clique hypothesis $\sf HC_{l}$. 
Proof of Theorem~\ref{thm:stat-bdr} combines the result of Lemma~\ref{lma:stat-alg} and Lemma~\ref{lma:info-low}.

\bibliographystyle{apalike}
\bibliography{bibfile}
\newpage

\begin{appendix}
\section{Convex Relaxation Algorithm }
\label{sec:conv}
In this section we will investigate a convex relaxation approach to the problem. The same algorithm has also been investigated in a parallel work of \cite{chen2014statistical}. Our analysis is slightly different, with the explicit construction of the dual certificate using the idea in \cite{gross2011recovering}. For the purposes of comparing to the spectral approach, we include the convex relaxation analysis in this section. Let us write the optimization problem
\begin{align*}
\min_{u \in \mathbb{R}^m,v \in \mathbb{R}^n} & \quad \| X - \lambda u v^T \|_F^2 \\
{\sf s.t.}                                   & \quad  \| u \|_{\ell_0} = k_m, \| v \|_{\ell_0} = k_n \\
                                             & \quad u \in \{ 0 , 1 \}^m, v \in \{ 0, 1 \}^n.
\end{align*}
This problem is non-convex: the feasibility set is non-convex, and so is the optimization function (although it is bi-convex). However, we can relax the problem and transform it into a convex optimization problem. Of course, we need to ensure that the solution to the relaxed problem is the exact solution (with high probability) under appropriate conditions.

The matrix version of the submatrix problem suggests that the signal matrix is of the structure ``low rank and sparsity on the singular vectors.'' We recall from the low rank matrix recovery literature, see e.g.  \cite{candes2011tight} and \cite{cai2014geometrizing}, that we can utilize the low rank structure and solve  relaxed versions as follows.

\paragraph{Relaxation 1}
Consider the constraint minimization relaxation,
\begin{align*}
\min_{M \in \mathbb{R}^{m\times n}} & \quad \| M \|_* \\
{\sf s.t.}                          & \quad  \| X - \lambda M \|_2 \leq  C \cdot \sigma (\sqrt{m}+ \sqrt{n}).
\end{align*}
Unfortunately, Relaxation 1 is only good in terms of estimation of the whole matrix. The stronger objective of localization requires simultaneous exploitation of sparsity and low rank-ness, as in Relaxation 2.

\paragraph{Relaxation 2}

Let us expand the objective of the original non-convex optimization problem, drop the quadratic term to make the procedure adaptive in terms of $\lambda$, and convexify the feasibility set at the same time.

\begin{algorithm}[H]
\KwIn{$X \in \mbb{R}^{m \times n}$ the data matrix. Size of the submatrix $k_m \times k_n$.}
\KwOut{A subset of the row indexes $\hat{R}_m$ and a subset of column indices $\hat{C}_n$ as the localization sets of the submatrix.}
    1. Solve the following convex optimization problem
	\begin{align*}
		\hat{M}_0 = \argmin{M \in \mathbb{R}^{m \times n}} & \quad - \langle X, M \rangle \\
                                                 \text{s.t.} & \quad  0 \leq M \leq 1_m 1_n^T \\
                                                            & \quad \| M \|_* \leq 1 \\
                                                            & \quad \langle M, 1_m 1_n^T \rangle  = k_m k_n
	\end{align*}
	2. Perform SVD on $\hat{M}_0$, denote the set of the non-zero entries on the top left singular vector to be $\hat{R}_m$ and the set of the non-zero entries on the top right singular vector to be $\hat{C}_n$.
\caption{Convex Relaxation Algorithm}
\label{CR1.Alg}
\end{algorithm}

The time complexity to solve this convex optimization problem is at least $\Theta((m+n)^3)$ implemented with alternating direction methods of multipliers (ADMM). The disadvantage is that the theoretical guarantee only holds for the exact solution $\hat{X}_0$; however, in reality we can only approximately find $\hat{X}_0$ through ADMM or some other optimization methods. We also remark that this algorithm requires the prior knowledge of the submatrix size $k_m, k_n$, which means it is not fully adaptive.

\begin{lemma}[Guarantee for Relaxation Algorithm]
	\label{lma:cv-rlx}
	Consider the submatrix model \eqref{submat.Mat.Frm} and the Algorithm \ref{CR1.Alg}. There exists a universal $C>0$ such that when 
	$$
	\frac{\lambda}{\sigma} \geq C \cdot \left( \sqrt{\frac{m \vee n}{k_m k_n}} + \sqrt{\frac{\log n}{k_m} \vee \frac{\log m}{k_n}} \right),
	$$
	the convex relaxation succeeds (in the sense that $\hat{R}_m = R_m, \hat{C}_n = C_n$) with probability at least $1 - 2 m^{-c} - 2 n^{-c} - 2 (mn)^{-c} - 2 \exp\left(-c(m+n)\right)$.
\end{lemma}

\begin{proof}[Proof of Lemma \ref{lma:cv-rlx}]
Let us construct the dual certificate to secure that the true solution is the unique solution. If we can construct a pair of a primal certificate $M^* = \sqrt{k_m k_n} U  V^T$  and dual certificate $\Delta^*, \Theta^* , \mu^*, \nu^*$ that satisfy 
\begin{align}
\label{dc1.eq}     X = -\Delta^* + \Theta^* + \mu^* (U V^T + W) + \nu^* 1_m 1_n^T \\
\label{cs.eq} \Delta^*_{ij} M^*_{ij} = 0, \Theta^*_{ij} (1  - M^*_{ij}) = 0, ~~ 1\leq i\leq m, 1\leq j               \leq n \\
\label{dc2.eq} \mu^* > 0 .
\end{align}
Here $W \in \mathcal{P}_{U^\perp \cap V^\perp}, \| W \|_2 \leq 1$, and $U V^T + W$ denotes the sub differential of $\| \cdot \|_*$ evaluated at $M^*$
$$
\partial|_{M = M^*} \| M \|_* = UV^T + W .
$$
Equation \eqref{cs.eq} is equivalent to 
\begin{align}
 \Delta^*_{ij} > 0,~~ \Theta^*_{ij} = 0, ~~ i \notin R_m ~\cup~ j \notin C_n \\
 \Theta^*_{ij} > 0,~~ \Delta^*_{ij} = 0, ~~ i \in R_m ~\cap~ j\in C_n .
\end{align}
We claim that any solution $\hat{M}$ to Relaxation 2 must satisfy $\hat{M} = M^*$. If not, 
we write $\hat{M} = M^* + H$ and find that 
\begin{align}
\label{cd.eq} 0 \leq M^* + H \leq 1_m 1_n^T \\
\| M^* + H \|_* \leq \| M^* \|_* \Rightarrow \langle H,  UV^T+W  \rangle \leq 0 \\
\langle H, 1_m 1_n^T \rangle = 0.
\end{align}
All of the above equations are due to the primal feasibility, and the second inequality also uses the convexity of $\| \cdot \|_*$. Note that \eqref{cd.eq} can be written in a more explicit form
\begin{align}
 0 \leq H_{ij} \leq 1, ~~ i \notin R_m ~\cup~ j \notin C_n \\
 -1 \leq H_{ij} \leq 0, ~~ i \in R_m ~\cap~ j\in C_n. 
\end{align}
Due to the optimality of $\hat{M}$ in terms of the objective function
$$
\langle X, M^* +H \rangle \geq \langle X, M^*  \rangle
$$
which means 
\begin{align}
0 \leq & \langle X, H \rangle  \\
 = & \langle -\Delta^* + \Theta^* + \mu^* (U V^T + W) + \nu^* 1_m 1_n^T  ,H \rangle  \\
 \leq  & \langle -\Delta^* + \Theta^*   , H \rangle .
\end{align}
We can see that if $H_{ij}>0$, then we must have $M^*_{ij}=0$, which through complimentary slackness implies  $\Theta^*_{ij} = 0 $. This in turn means $ - \Delta_{ij} < 0$. When $H_{ij}<0$, we must have $M^*_{ij} = 1$, which again means $\Delta_{ij} = 0$, $\Theta^*_{ij} > 0$, a contradiction. Thus $H_{ij} =0$ for all $i,j$.

The properties we impose on the dual certificates are motivated from the Karush-Kuhn-Tucker (KKT) conditions.
Introduce the dual variables $\Delta,\Theta \in \mathbb{R}_+^{m \times n}$, $\mu \in \mathbb{R}_+, \nu \in \mathbb{R}$ for the four feasibility conditions. Then the Lagrangian is
\begin{align*}
\mathcal{L}(M,\Delta,\Theta,\mu,\nu) = - \langle X, M \rangle  - \langle \Delta, M \rangle + \langle \Theta, M -1_m 1_n^T \rangle + \mu (\| M \|_* -1   )+ \nu \left(\langle M, 1_m 1_n^T \rangle  - k_m k_n\right).
\end{align*}
The associated KKT conditions are
\begin{align}
\label{kkt.eq}
X = -\Delta + \Theta + \mu (U_M V_M^T + W_M) + \nu 1_m 1_n^T \\
0 \leq M \leq 1_m 1_n^T \\
 \| M \|_* \leq 1 \\
 \langle M, 1_m 1_n^T \rangle = k_m k_n \\
 \Delta \geq 0, \Theta \geq 0, \mu \geq 0 \\
 \Delta_{ij} M_{ij} = 0, \Theta_{ij} (1  - M_{ij}) = 0.
\end{align}

Now let us see how to construct the dual certificate $-\Delta^* + \Theta^*$, $\mu^*, \nu^*$ to satisfy the conditions \eqref{dc1.eq} - \eqref{dc2.eq}. Expand \eqref{kkt.eq} as
\begin{align}
	\label{kkt2.eq}
 -\Delta^* + \Theta^* = \lambda 1_{R_m} 1_{C_n}^T + \mathcal{P}_{U \cup V}(Z) + \mathcal{P}_{U^T \cap V^T}(Z) -\mu^* \left(\frac{1}{\sqrt{k_m k_n}} 1_{R_m} 1_{C_n}^T + W \right) - \nu^* 1_m 1_n^T
\end{align}
where
\begin{align}
\mathcal{P}_{U^T \cap V^T}(Z) = Z - \mathcal{P}_{U \cup T}(Z).
\end{align}
Choose
\begin{align}
W = \frac{1}{\mu^*} \mathcal{P}_{U^T \cap V^T} (Z).
\end{align}
Thus if we choose $\mu^* \geq  C \cdot \sigma(\sqrt{m}+\sqrt{n})$, it holds that
\begin{align*}
\| W \|_2 = \frac{1}{\mu^*} \| \mathcal{P}_{U^T \cap V^T} (Z) \|_2 \leq  \frac{1}{\mu^*}  \| Z \|_2 \leq 1
\end{align*}
with probability at least $1- 2 \exp\left( -c(m+n) \right)$.
Thus with the choice of $W$, Equation \eqref{kkt2.eq} becomes
\begin{align}
	 -\Delta^* + \Theta^* = \lambda 1_{R_m} 1_{C_n}^T + \mathcal{P}_{U \cup V}(Z) - \mu^* \frac{1}{\sqrt{k_m k_n}} 1_{R_m} 1_{C_n}^T - \nu^* 1_m 1_n^T.
\end{align}
Hence, we need to have
\begin{align}
	\label{ld1.eq}
\lambda - \mu^* \frac{1}{\sqrt{k_m k_n}} - \nu^* - \max_{ij} |[\mathcal{P}_{U \cup V}(Z)]_{ij}| > 0, ~~ i \in R_m ~\cap~ j\in C_n\\
    \label{ld2.eq}
	\max_{ij} |[\mathcal{P}_{U \cup V}(Z)]_{ij}| - \nu^* < 0, ~~  i \in R_m^c ~\cup~ j\in C_n^c.
\end{align}

Now let us write out the explicit form of the projection $\mathcal{P}_{U \cup V}(Z)$
\begin{align}
\mathcal{P}_{U \cup V}(Z) = \frac{1}{k_m} 1_{R_m} 1_{R_m}^T Z + \frac{1}{k_n} Z 1_{R_n} 1_{R_n}^T - \frac{1}{k_m k_n} 1_{R_m} 1_{R_m}^T Z 1_{R_n} 1_{R_n}^T.
\end{align}
Let us see the concentration property of $[\mathcal{P}_{U \cup V}(Z)]_{ij}$:
\begin{align}
	[\mathcal{P}_{U \cup V}(Z)]_{ij} &= \frac{1}{k_m} \left( \sum_{k \in R_m} Z_{k j} \right) 1_{i \in R_m} + \frac{1}{k_n} \left( \sum_{l \in C_n} Z_{i l} \right) 1_{j \in C_n} -  \frac{1}{k_m k_n} \left( \sum_{k \in R_m, l \in C_n} Z_{i j} \right) 1_{i \in R_m, j\in C_n} \\
	|[\mathcal{P}_{U \cup V}(Z)]_{ij}| &\leq \left|\frac{1}{k_m} \sum_{k \in R_m} Z_{k j}  \right| + \left|\frac{1}{k_n} \sum_{l \in C_n} Z_{i l} \right| + \left|\frac{1}{k_m k_n} \sum_{k \in R_m, l \in C_n} Z_{i j} \right|.
\end{align}
For all the $1\leq j \leq n$
\begin{align}
	\max_{1 \leq j \leq n} \left|\frac{1}{k_m} \sum_{k \in R_m} Z_{k j}  \right| \leq \sqrt{2(1+1/c)} \cdot \sigma \sqrt{\frac{\log n}{k_m}}
\end{align}
with probability at least $1 - 2 n^{-c}$.
For all the $1\leq i \leq m$
\begin{align}
	\max_{1 \leq i \leq m} \left|\frac{1}{k_n} \sum_{l \in C_n} Z_{i l} \right| \leq \sqrt{2(1+1/c)} \cdot \sigma \sqrt{\frac{\log m}{k_n}}
\end{align}
with probability at least $1 - 2 m^{-c}$. For all $i,j$,
\begin{align}
	\max_{i,j}\left|\frac{1}{k_m k_n} \sum_{k \in R_m, l \in C_n} Z_{i j} \right| \leq \sqrt{2(1+1/c)} \cdot \sigma \sqrt{\frac{\log (mn)}{k_m k_n}}
\end{align}
with probability at least $1 - 2 (mn)^{-c}$. Now, pick the dual certificate variables in the following way
\begin{align}
	\mu^* & = C \cdot \sigma( \sqrt{m}+\sqrt{n} ) \\
	\nu^* & = C \cdot \sigma \left( \sqrt{\frac{\log m}{k_n}} +  \sqrt{\frac{\log n}{k_m}} \right) \\
	\Theta^* & > 0, ~~ i \in R_m ~\cap~ j\in C_n \\
	\Delta^* & < 0, ~~ i \in R_m^c ~\cup~ j\in C_n^c
\end{align}
where the last two equation follows from \eqref{ld1.eq} and \eqref{ld2.eq}.
We conclude that the relaxation algorithm succeeds with probability at least $$1 - 2 m^{-c} - 2 n^{-c} - 2 (mn)^{-c} - 2 \exp(-c(m+n))$$
if
$$
\lambda \geq C \cdot 
\sigma \left(\sqrt{\frac{m+n}{k_m k_n}} +  \sqrt{\frac{\log m}{k_n}}  + \sqrt{\frac{\log n}{k_m}}  \right).
$$
We have achieved the the same boundary as the spectral method upper bound.
\end{proof}

\section{Algorithmic Reduction for Detection}
\label{sec:complow-det}
\begin{theorem}[Computational Lower Bounds for Detection]
	\label{CLD.Thm}
	Consider the submatrix model \eqref{submat.Mat.Frm} with parameter tuple $(m = n, k_m \asymp k_n \asymp n^{\alpha}, \lambda/\sigma = n^{-\beta})$, where $\frac{1}{2}<\alpha<1,~ \beta>0$. Under the hardness assumption $\sf HC_{d}$, if 
	\begin{align*}
		\frac{\lambda}{\sigma} \precsim \frac{m+n}{k_m k_n} ~~ \Rightarrow ~~ \beta > 2\alpha - 1,
	\end{align*}
	it is not possible to detect the true support of the submatrix with probability going to $1$ for any polynomial algorithm.
\end{theorem}

\begin{proof}[Proof of Theorem \ref{CLD.Thm}]
We would like to build a randomized polynomial mapping from the hidden clique graph $\mathcal{G}(N, \kappa(N))$ to a matrix $M(m=n,k_m \asymp k_n \asymp k,\lambda/\sigma)$ for the submatrix model. Denote this transformation as $$\mathcal{T}: \mathcal{G}(N, \kappa(N)) \rightarrow M(n,k = n^{\alpha},\lambda/\sigma = n^{-\beta}).$$ 
 
 There are several stages of the construction. First, we define a graph that is stochastically equivalent to the hidden clique graph $\mathcal{G}$, but is easier for the analysis. Let us call it $\mathcal{G}^{e}$.  $\mathcal{G}^e$ has the property: each node independently has the probability $\kappa(N)/N$ to be a clique node. By Bernstein's inequality, with probability at least $1 - 2N^{-1}$, the number of cliques $\kappa^{e}$ in $\mathcal{G}^e$
 \begin{align}
 \kappa \left(1 - \sqrt{\frac{4 \log N}{\kappa}} \right) \leq \kappa^e \leq  \kappa \left(1 + \sqrt{\frac{4 \log N}{\kappa}} \right) \Rightarrow \kappa^e \asymp \kappa
 \end{align}
 as long as $\kappa \succsim \log N$. 
 
 Consider a double sized hidden clique graph $\mc{G}^e(2N, 2\kappa(N))$ with $N = n^{1+\beta}$, and $\kappa(N) = k = n^{\alpha}$, $\frac{1}{2}<\alpha<1$. Denote the clique nodes set as $C_{N,\kappa}$. Connect the hidden clique graph to form a symmetric matrix $G \in \{-1,1\}^{2N \times 2N}$, where $G_{ij} = 1$ if $i,j \in C_{N,\kappa}$, otherwise with equal probability to be either $-1$ or $1$. Take out the upper-right submatrix of $G$, $G_{UR}$ where $U$ is the index set $1\leq i \leq N$ and $R$ is the index set $N+1 \leq j \leq 2N$.

 Partition the rows of $G_{UR} \in \mathbb{R}^{n^{1+\beta} \times n^{1+\beta}}$, to form $n$ blocks. The $s$'s block, $1 \leq s \leq n$, corresponds to the row index set $I_s  = \{i:(s-1)n^{\beta}+1 \leq i \leq  s n^{\beta}\}$. Construct the $M \in \mathbb{R}^{n \times n}$ matrix in the following way
 \begin{align}
	\mathcal{T}:~~ M_{st} = \frac{1}{n^{\beta}} \sum_{i \in I_s, j\in I_t} (G_{UR})_{ij}, ~~ 1\leq s,t \leq n.
 \end{align}
 There are two cases, if $I_{s} \cap C_{N,\kappa} \neq \emptyset$ and $I_{t} \cap C_{N,\kappa} \neq \emptyset$, namely when $s$'s block contains at least $1$ clique node, and so does $t$'s block, then $\mathbb{E} M_{st} \geq n^{-\beta}$. Otherwise $\mathbb{E} M_{st}  = 0$. Note that for the index $s,t$ that has no clique inside, $\mathbb{E} M_{st}$ is sub-Gaussian random variable, due to Hoeffding's inequality
 \begin{align}
	 \mathbb{P}\left(|M_{st} - \mathbb{E}M_{st}| \geq u \right) \leq 2 \exp (- u^2/2).
 \end{align}
 For the $s,t$ with clique nodes inside, we know the maximum number of clique nodes is $\log n$ (due to Bernstein's inequality that $\max_{1\leq s\leq n} |I_{s} \cap C_{N,\kappa}| \leq \frac{k}{n} + \frac{8}{3} \log n$ with probability at least $1- n^{-1}$), which means there are at least $n^{\beta} - \frac{k}{n} - \frac{8}{3} \log n$ many independent Rademacher random variables in each $s,t$ block, thus
 \begin{align}
	 \mathbb{P}\left(|M_{st} - \mathbb{E}M_{st}| \geq u \right) \leq 2 \exp \left(- (1 - C \cdot (k n^{-1-\beta} + n^{-\beta} \log n))u^2/2 \right).
 \end{align}
 Thus we can take sub-Gaussian parameter to be any $\sigma < 1$ because $k n^{-1-\beta}, n^{-\beta} \log n$ are both $o(1)$. Now this constructed $M(n,k)$ matrix satisfies the submatrix model with $\lambda = n^{-\beta}$ and sub-Gaussian parameter $\sigma = 1 - o(1)$. 
 
 Let us see how many elements in $1\leq s\leq n$ are such that $I_{s} \cap C_{N,\kappa} \neq \emptyset$. Namely, we want to estimate how many clique nodes there exist in the transformed submatrix model. We have the two sided bound
 	 $$
 	 k \left(1 - \frac{1}{2n^{1-\alpha}} \right)  \leq \mathbb{E} |\{s: I_{s} \cap C_{N,\kappa} \neq \emptyset, 1\leq s \leq n \} | \leq  k.
 	 $$
 	 which is of the order $k$. Using Bernstein's bound, we have with high probability 
	 $$
	 k\left(1 - \frac{1}{2n^{1-\alpha}} \right) \left(1 - \sqrt{\frac{4\log n}{k}} \right) < |\{s: I_{s} \cap C_{N,\kappa} \neq \emptyset, 1\leq s \leq n \} | < k \left(1 + \sqrt{\frac{4\log n}{k}} \right)
	 $$
	 Thus the submatrix model $M(m=n,k_m \asymp k_n,\lambda/\sigma)$ satisfies $k_m \asymp k_n \asymp k$.
 
 Suppose there exists an polynomial time algorithm $\mathcal{A}_M$ that pushes below the computational boundary quantitatively by a small $\epsilon>0$ amount
 \begin{align}
 n^{-\beta} \asymp \frac{\lambda}{\sigma} \precsim  \frac{m+n}{k_m k_n}  \asymp n^{1-2\alpha} \Rightarrow \beta > 2\alpha-1.
 \end{align}
 where $\beta = 2\alpha-1+\epsilon$.
 Namely, $\mathcal{A}_M$ detects below the boundary $\sigma \frac{m+n}{k_m k_n}$, then it naturally introduced a polynomial time detection algorithm for hidden clique problem $\mc{G}(N = n^{1+\beta}, \kappa = n^{\alpha})$, which violates the $\sf HC_{d}$ because
 $$
\kappa(N) \asymp N^{\frac{\alpha}{1+\beta}} \asymp N^{\frac{\alpha}{2\alpha+\epsilon}}\precsim N^{\frac{1}{2}} .
 $$
\end{proof}

\end{appendix}

\end{document}